\documentclass[12pt]{amsart} 

\usepackage{amssymb}
\usepackage{amsmath}
\usepackage{amsthm}
\usepackage{verbatim}
\usepackage{url}
\usepackage{color}
\usepackage{mathabx}

\newcommand{\Ccal}{{\mathcal C}}
\newcommand{\cC}{{\mathcal C}}
\newcommand{\Dcal}{{\mathcal D}}
\newcommand{\cD}{{\mathcal D}}

\newcommand{\CC}{\mathbb{C}}

\newcommand{\NN}{\mathbb{N}}

\newcommand{\RR}{\mathbb{R}}

\def\onto{\twoheadrightarrow}           % a surjection onto
\newcommand{\mcm}[3]{\newcommand{#1}[#2]{{\ensuremath{#3}}}} 
\mcm{\restric}{0}{\upharpoonright}

\DeclareMathOperator{\Fuzz}{Fuzz}
\DeclareMathOperator{\FuzzGrRings}{Fuzz^\prime}

\newcommand\N{{\mathbb N}}

\numberwithin{equation}{section}

\newtheorem{theorem}[equation]{Theorem}
\newtheorem{lemma}[equation]{Lemma}

\newtheorem{proposition}[equation]{Proposition}

\newtheorem{corollary}[equation]{Corollary}
\newtheorem{cor}[equation]{Corollary}
\newtheorem*{theorem*}{Theorem}

\theoremstyle{definition}
\newtheorem{defn}[equation]{Definition}
\newtheorem{example}[equation]{Example}

\theoremstyle{remark}
\newtheorem{remark}[equation]{Remark}

\newcommand{\involution}{\overline}
\newcommand{\noinvolution}{}

\begin{document}
\title{Matroids over Partial Hyperstructures}
\author{Matthew Baker}
\email{mbaker@math.gatech.edu}
\address{School of Mathematics,
          Georgia Institute of Technology, USA}
%          Atlanta GA 30332-0160, USA}
\author{Nathan Bowler}
\email{Nathan.Bowler@uni-hamburg.de}
\address{Department of Mathematics, Universit{\"a}t Hamburg, Germany}

\date{\today}

\thanks{The first author's research was supported by the National Science Foundation research grant DMS-1529573.}

\begin{abstract}
We present an algebraic framework which simultaneously generalizes the notion of linear subspaces, matroids, valuated matroids, oriented matroids, 
and regular matroids.  
To do this, we first introduce algebraic objects which we call {\em tracts}; they generalize both hyperfields in the sense of Krasner and partial fields in the sense of Semple and Whittle.
We then define matroids over tracts; in fact, there are (at least) two natural notions of matroid in this general context, which we call {\em weak} and {\em strong} matroids.
We give ``cryptomorphic'' axiom systems for such matroids in terms of circuits, Grassmann-Pl{\"u}cker functions, and dual pairs, and establish some 
basic duality results.
We then explore sufficient criteria for the notions of weak and strong matroids to coincide.  This is the case whenever vectors and covectors are orthogonal, and is closely related to the notion of ``perfect fuzzy rings'' from \cite{DressWenzelPM}.
% \cite{DW_Perfect}.
For example, if $F$ is a particularly nice kind of tract called a doubly distributive partial hyperfield, we show that the notions of weak and strong matroid over $F$ coincide.
Our theory of matroids over tracts is closely related to but more general than
% as well as phased matroids in the sense of Anderson-Delucchi.  
``matroids over fuzzy rings'' in the sense of Dress and Dress--Wenzel \cite{Dress,DressWenzelGP,DressWenzelVM,DressWenzelPM}.
% We call the resulting objects {\em matroids over hyperfields}.  
\end{abstract}

\maketitle

\section{Introduction} \label{sec:intro}

Matroid theory is a remarkably rich part of combinatorics with links to algebraic geometry, optimization, and many other areas of mathematics. Matroids provide a useful abstraction of the notion of linear independence in vector spaces, and can be thought of as combinatorial analogues of linear subspaces of $K^m$, where $K$ is a field.
A key feature of matroids is that they possess a duality theory which abstracts the concept of orthogonal complementation from linear algebra.
There are a number of important enhancements of the notion of matroid, including oriented matroids, valuated matroids, and regular matroids.
% and phased matroids in the sense of Anderson-Delucchi.  
In this paper, we provide an algebraic framework for unifying all of these enhancements, introducing what we call {\bf matroids over tracts}.
% \footnote{{\bf Matt:} I'm not crazy about Laura's suggestion of ``swamps'', I prefer ``tracts''.  But I'd be open to calling them ``partial hyperfields'' instead if you think that's a better name.  We may wish to reserve that terminology, however, for the objects defined in \S\ref{sec:partialhyperfieldtract}.}
Examples of tracts include hyperfields\footnote{For the reader's convenience, there is a self-contained version of the present paper written in the more specialized language of hyperfields available on the arXiv, see \cite{BakerBowlerHyperfield}.} in the sense of Krasner, partial fields in the sense of Semple and Whittle \cite{SempleWhittle}, and fuzzy rings in the sense of Dress \cite{Dress}, cf.~\S\ref{sec:fuzzyringtract} below.

\medskip

It turns out that there are (at least) two natural notions of matroids over a tract $F$, which we call {\bf weak $F$-matroids} and {\bf strong $F$-matroids}.
% \footnote{In arXiv versions 1 through 3 of the present paper, the first author incorrectly claimed that weak and strong matroids coincide over all hyperfields.  See \S\ref{sec:meaculpa} for a discussion of this error and how it has been rectified in the present version.}  
In this paper we give ``cryptomorphic'' axiom systems for both kinds of $F$-matroids and present examples showing that the two notions of $F$-matroid diverge for certain tracts (which can be taken to be hyperfields) $F$.  
On the other hand, if $F$ is a {\bf doubly distributive partial hyperfield}, we show that the notions of weak and strong $F$-matroid coincide.

\subsection{Tracts}

% Instead of fuzzy rings, o
Our basic algebraic object in this paper is what we call a {\bf tract}.
% \footnote{{\bf Matt:} Explain that while in practice most tracts satisfy stronger properties, e.g. $N_G$ is an ideal, tracts appear to be the natural setting in which our proofs make sense.}  
A tract is an abelian 
%(not necessarily abelian)\footnote{At the moment, the proofs in this paper are written assuming that $G$ is abelian.  In the near future we will rewrite the proofs so that commutativity is not used.} 
group $G$ (written multiplicatively), together with an {\bf additive relation structure} on $G$, which is a subset $N_G$ of the group semiring $\NN[G]$ satisfying:
\begin{itemize}
\item[(T0)] The zero element of $\NN[G]$ belongs to $N_G$.
\item[(T1)] The identity element $1$ of $G$ is not in $N_G$.
\item[(T2)] There is a unique element $\epsilon$ of $G$ with $1 + \epsilon \in N_G$.
\item[(T3)] $N_G$ is closed under the natural action of $G$ on $\NN[G]$.
\end{itemize}  

One thinks of $N_G$ as those linear combinations of elements of $G$ which ``sum to zero'' (the $N$ in $N_G$ stands for ``null set'').

\medskip

We let $F = G \cup \{ 0 \} \subset \NN[G]$, and we often refer to the tract $(G,N_G)$ simply as $F$.  
(This is similar to thinking of a field $K$ set-theoretically as its multiplicative group $K^\times$ together with an additional element called $0$.)
We will sometimes write $F^\times$ instead of $G$.

\medskip

\begin{lemma} \label{lem:negatives}
Let $F=(G,N_G)$ be a tract.
\begin{itemize}
\item[(a)] If $x,y \in G$ satisfy $x+y \in N_G$, then $y = \epsilon x$.  
\item[(b)] $\epsilon^2 = 1$.
%  and $\epsilon$ belongs to the center of $G$.
\item[(c)] $G \cap N_G = \emptyset$.
\end{itemize}
\end{lemma}

\begin{proof}
For (a), we have $(x+y) x^{-1} = 1 + yx^{-1} \in N_G$ so $yx^{-1} = \epsilon$ by (T2).  Thus $y = \epsilon x$.

For (b), apply (a) with $x=\epsilon$ and $y=1$ to the identity $1 + \epsilon \in N_G$.
% For the second assertion, note that if $x \in G$ then $x + x \epsilon = x(1 + \epsilon) \in N_G$.  By (a) 
% we have $x \epsilon = \epsilon x$.

For (c), note that if $g \in G \cap N_G$ then by (T3) $g^{-1}g = 1 \in N_G$, contradicting (T1).
\end{proof}

Because of Lemma~\ref{lem:negatives}, we often write $-1$ instead of $\epsilon$ and $-x$ instead of $\epsilon x$.
%  when this will not cause  any confusion.

%\medskip

% {\bf Convention:} For notational convenience, for $x,y \in G$ we define 
% \begin{equation} \label{eq:convention}
% \frac{x}{y} := y^{-1}x.
% \end{equation}

\medskip

A {\bf homomorphism} $f : (G,N_G) \to (G',N_{G'})$ of tracts is a group homomorphism $f : G \to G'$, together with a map $f : \NN[G] \to \NN[G']$ satisfying $f(\sum a_i g_i) = \sum a_i f(g_i)$ for $a_i \in \NN$ and $g_i \in G$, such that if $\sum a_i g_i \in N_G$ then
$\sum a_i f(g_i) \in N_{G'}$.

\subsection{Examples} \label{sec:tractexamples}
Tracts are extremely flexible objects, as we will see in Section \ref{sec:examplestracts}. We will see there that they generalize hyperfields and partial fields, as well as a common generalization of the two which we call partial hyperfields. They also generalize the fuzzy rings of Dress and Wenzel.

A {\em hyperfield} is an algebraic structure akin to a field with 1, but where addition is allowed to be multivalued. (Multivalued algebraic operations might seem exotic, but in fact hyperrings and hyperfields appear quite naturally in a number of mathematical settings and their properties have been explored by numerous authors in recent years.)  There is still a notion of additive inverse, but rather than requiring that $x$ plus $-x$ equals $0$,
one merely assumes that $0$ {\em belongs to the set} ``$x$ plus $-x$''. 

The notion of {\em partial field} was introduced by Semple and Whittle in \cite{SempleWhittle} as an algebraic framework for unifying various matroid representation theorems due to Tutte and Whittle.  It was further developed by Pendavingh and van Zwam in \cite{PvZLifts,PvZSkew}. {\em Fuzzy rings} were developed by Dress and Wenzel as an alternative algebraic framework for representing matroids. 

For each of these kinds of object we may define a corresponding tract, in such a way that representability over that tract is the same as representability over the original object (for hyperfields we take this as our definition of representability over the hyperfield, since this was not previously defined).

We can recover many familiar classes of matroids by considering representability over particular hyperfields. In Section \ref{sec:examplestracts} we will define certain hyperfields $\mathbb{K}$ (the {\em Krasner hyperfield}), $\mathbb{S}$ (the {\em hyperfield of signs}) and $\mathbb{T}$ (the {\em tropical hyperfield}). The classes of matroid representations over these are, respectively, all matroids, oriented matroids, and valuated matroids.

\subsection{Cryptomorphic axiomatizations}

Matroids famously admit a number of ``cryptomorphic'' descriptions, meaning that there are numerous axiom systems for them which turn out to be non-obviously equivalent. 
Two of the most useful cryptomorphic axiom systems for matroids (resp. oriented, valuated) are the descriptions in terms of {\em circuits} (resp. signed, valuated circuits) and {\em basis exchange axioms} (resp. chirotopes, valuated bases).  A third (less well-known but also very useful) cryptomorphic description in all of these contexts involves {\em dual pairs}.  We generalize all of these cryptomorphic descriptions (for both weak and strong matroids over tracts) with a single set of theorems and proofs.  
% The proof of the duality theorem utilizes the equivalence of these different descriptions.

The circuit description of strong (resp.~weak) matroids over tracts is a bit technical to state, see \S\ref{sec:MatroidsOverTracts} for the precise definition.  Roughly speaking, though, if $F=(G,N_G)$ is a tract, a subset $\cC$ of $F^m$ not containing the zero-vector is the set of {\bf $F$-circuits of a weak matroid with coefficients in $F$}  if it is stable under scalar multiplication, satisfies a support-minimality condition, and obeys a {\em modular elimination law}.  (The {\bf support} of $X \in \cC$ is the set of all $i$ such that $X_i \neq 0$.)  
The ``modular elimination'' property means that
if the supports of $X,Y \in \cC$ are ``sufficiently close'' (in a precise poset-theoretic sense) and $X_i = -Y_i$ for some $i$, then one can find a ``quasi-sum'' $Z \in \cC$ with $Z_i=0$ and $X_j + Y_j - Z_j \in N_G$ for all $j$.
% $Z_j \in X_j \boxplus Y_j$ for all $j$.
The underlying idea is that the $F$-circuits of an $F$-matroid behave like the set of support-minimal nonzero vectors in a linear subspace of a vector space.
The most subtle part of the definition is the restriction that the supports of $X$ and $Y$ be sufficiently close; this restriction is not encountered ``classically'' when working with matroids, oriented matroids, or valuated matroids, but it is necessary in the general context in which we work, as has already been demonstrated by Anderson and Delucchi in their work on phased matroids \cite{AndersonDelucchi}.  They give an example of a phased matroid which satisfies modular elimination but not a more robust elimination property.  
In  \S\ref{sec:MatroidsOverTracts} we also present a stronger and somewhat more technical set of conditions characterizing 
the set of $F$-circuits of a {\bf strong} $F$-matroid.

In the general context of matroids over tracts, the simplest and most useful way to state the ``basis exchange'' or chirotope / phirotope axioms is in terms of what we call {\em Grassmann-Pl{\"u}cker functions}.  
A nonzero function $\varphi : F^r \to F$ is called a {\bf Grassmann-Pl{\"u}cker function} if it is alternating and satisfies (tract analogues of) the basic algebraic identities satisfied by the determinants of the $(r \times r)$-minors of an $r \times m$ matrix of rank $r$ (see \S\ref{sec:GPsection} for a precise definition).
By a rather complicated argument, the definition of strong $F$-matroids in terms of strong $F$-circuits turns out to be cryptomorphically equivalent to the definition in terms of Grassmann-Pl{\"u}cker functions.
% One can think of a Grassmann-Pl{\"u}cker function as a point on a hyperfield-scheme analogous to the Grassmannian variety $G(r,m)$ (cf.~\S\ref{sec:Dressian}).
We also define {\em weak Grassmann-Pl{\"u}cker functions} and relate them to weak $F$-circuits.

The ``dual pair'' description of $F$-matroids is perhaps the easiest one to describe in a non-technical way, assuming that one already knows what a matroid is.  If $\underline{M}$ is a matroid in the usual sense, we call a subset $\cC$ of $F^m$ not containing $0$ and closed under nonzero scalar
multiplication an {\bf $F$-signature} of $\underline{M}$ if the support of $\cC$ in $E=\{ 1,\ldots, m \}$ is the set of circuits of $\underline{M}$.  The {\bf inner product} of two vectors $X,Y \in F^m$ 
% with respect to a given involution $x \mapsto \overline{x}$ on $F$ 
is $X \cdot Y := \sum_{i=1}^m X_i Y_i$, 
% $X \cdot Y := \bigboxplus_{i=1}^m X_i \cdot \involution{Y}_i$, 
and we call $X$ and $Y$ {\bf orthogonal} (written $X \perp Y$) if $X \cdot Y \in N_G$.
% $0 \in X \odot Y$.  
A pair $(\cC,\cD)$ consisting of an $F$-signature $\cC$ of $\underline{M}$ and an $F$-signature $\cD$ of the dual matroid $\underline{M}^*$ is called a {\bf dual pair} if $X \perp Y$ for all $X \in \cC$ and $Y \in \cD$.  By a rather complex chain of reasoning, it turns out that a strong $F$-matroid in either of the above two senses is equivalent to a dual pair $(\cC,\cD)$ as above.  We also define {\em weak dual pairs} and relate them to weak $F$-circuits and weak Grassmann-Pl{\"u}cker functions.

In the recent preprint \cite{AndersonVectors}, Laura Anderson proves that strong matroids over tracts can be characterized in terms of a cryptomorphically equivalent set of {\em vector axioms}.

\subsection{Duality}
% \subsection{Duality and functoriality}

If $\cC$ is the collection of strong $F$-circuits of an $F$-matroid $M$ and $(\cC,\cD)$ is a dual pair of $F$-signatures of the matroid $\underline{M}$ underlying $M$ (whose circuits are the supports of the $F$-circuits of $M$), it turns out that $\cD$ is precisely the set of (non-empty) support-minimal elements of the orthogonal complement of $\cC$ in $F^m$, and $\cD$ forms the set of $F$-circuits
of a strong $F$-matroid $M^*$ which we call the {\bf dual strong matroid}.
% $F^*$-circuits\footnote{Here $F^*$ denotes $F$ but with the opposite multiplication on the unit group $F^\times$.} of a strong $F^*$-matroid $M^*$ which we call the {\bf dual strong matroid}.

Duality behaves as one would hope: for example $M^{**}=M$, duality is compatible in the expected way with the notions of deletion and contraction, and the underlying matroid of the dual is the dual of the underlying matroid. 
There is a similar, and similarly behaved, notion of duality for weak $F$-matroids.

Matroids over tracts admit a useful push-forward operation: given a homomorphism of tracts $f : F \to F'$ and a strong (resp. weak) $F$-matroid $M$, there is an induced strong (resp. weak) $F'$-matroid $f_* M$ which can be defined using any of the cryptomorphically equivalent 
axiomatizations.  The ``underlying matroid'' construction coincides with the push-forward of an $F$-matroid $M$ to the Krasner hyperfield ${\mathbb K}$ (identified with the corresponding tract) via the canonical homomorphism $\psi : F \to {\mathbb K}$ sending $0$ to $0$ and every $g \in F^\times$ to $1$.  

If $\sigma : {\mathbb R} \to {\mathbb S}$ is the map taking a real number to its sign and $W \subseteq {\mathbb R}^m$ is a linear subspace (considered in the natural way as an ${\mathbb R}$-matroid), the push-forward $\sigma_*(W)$ coincides with the oriented matroid which one traditionally associates to $W$.  Similarly, if $v : K \to {\mathbb T}$ is the valuation on a non-Archimedean field and $W \subseteq K^m$ is a linear subspace,
$v_*(W)$ is just the {\bf tropicalization} of $W$ considered as a {\em valuated matroid} (cf.~\cite{MaclaganSturmfels}).  
% 
% There is a similar story for phased matroids and the natural ``phase map'' $p: {\mathbb C} \to {\mathbb P}$.
% If $\phi : K \to K'$ is an embedding of fields, the pushforward $\phi_*(M)$ of a $K$-matroid $M$ corresponding to a linear subspace $W \subseteq K^m$ is the $K'$-matroid corresponding to the linear subspace $W \otimes_{K} K' \subseteq (K')^m$.

\subsection{Relation to the work of Dress and Wenzel}

In \cite{Dress}, Andreas Dress introduced the notion of a {\bf fuzzy ring} and defined matroids over such a structure, showing that linear subspaces, matroids in the usual sense, and oriented matroids are all examples of matroids over a fuzzy ring.
In \cite{DressWenzelVM}, Dress and Wenzel introduced the notion of {\em valuated matroids} as a special case of matroids over a fuzzy ring.  
The results of Dress and Wenzel in \cite{Dress,DressWenzelGP,DressWenzelVM} include a duality theorem and a cryptomorphic characterization
of matroids over fuzzy rings in terms of Grassmann-Pl{\"u}cker functions.  (They also work with possibly infinite ground sets, whereas for simplicity we restrict ourselves to the finite case.)  

Our work generalizes theirs. In addition to the fact that tracts generalize fuzzy rings (see \S\ref{sec:fuzzyringtract} above),
we provide cryptomorphic characterizations of matroids in terms of circuits and dual pairs, which one does not find explicitly in the work of Dress--Wenzel.
Our work also has the advantage that (matroids over) tracts are arguably simpler and more intuitive to work with than (matroids over) fuzzy rings.

% In their recent preprint \cite{GiansiracusaJunLorscheid}, Jeff Giansiracusa, Jaiung Jun, and Oliver Lorscheid show that there is a fully faithful functor from hyperfields to fuzzy rings which induces an equivalence between the theory of strong matroids over a hyperfield and the theory of matroids over the corresponding fuzzy ring.  More precisely, their functor induces an equivalence of categories between hyperfields and {\em field-like} fuzzy rings which identifies strong matroids over the former with matroids over the latter.

{We use theorems of Dress and Wenzel from \cite{DressWenzelPM} to show that if $F$ is a {\em doubly distributive partial hyperfield} (or, more generally, a {\em perfect tract}, cf.~\S{\ref{sec:perfect}} for the definition), the notions of weak and strong $F$-matroid coincide.}

\subsection{Relation to the work of Anderson and Delucchi}

While the proofs of our main theorems are somewhat long and technical, in principle a great deal of the hard work has already been done in \cite{AndersonDelucchi}, so on a number of occasions we merely point out that a certain proof from \cite{AndersonDelucchi} goes through {\em mutatis mutandis} in the general setting of matroids over tracts.  (By way of contrast, the proofs in the standard works on oriented and valuated matroids tend to rely on special properties of the sign and tropical hyperfields which do not readily generalize.)
% That being said, there is a crucial gap in \cite{AndersonDelucchi} which makes some of the results of that paper incomplete; see the Appendix and \S\ref{sec:meaculpa} for details.

% \subsection{Relation to Lorscheid's blueprints}

\subsection{Other related work}

Despite the formal similarity in their titles, the theory in this paper generalizes matroids in a rather different way from the paper ``Matroids over a Ring'' by Fink and Moci \cite{FinkMoci}.  For example, if $K$ is any field, a matroid over $K$ in the sense of Fink--Moci is just a matroid in the usual sense (independent of $K$), while for us a matroid over $K$ is a linear subspace of $K^m$.  
% The work of Fink--Moci generalizes, among other things, the concept of {\em arithmetic matroids}, which we do not discuss.  

% It is possible that there is a Grand Unified Theory of matroids over hyperrings which encompasses both points of view, but we have not given this matter serious thought.  In principle, we could formulate many of the definitions in this paper over {\em hyperrings}, and not just over hyperfields, but we do not know how to say anything useful in this general context so we restrict ourselves to the case where we're actually able to prove theorems.  

The thesis of Bart Frenk \cite{Frenk} deals with matroids over certain kinds of algebraic objects which he calls tropical semifields; these are defined as sub-semifields of ${\mathbb R} \cup \{ \infty \}$.  Matroids over tropical semifields include, as special cases, both matroids in the traditional sense and valuated matroids, but not for example oriented matroids, linear subspaces of $K^m$ for a field $K$, or phased matroids.  Tropical semifields are a particular special case of idempotent semifields, and matroids over the latter are the subject of an interesting recent paper by the Giansiracusa brothers \cite{GiansiracusaGrassmann}.  
% They characterize matroids over idempotent semifields in a way which seems unlikely to generalize to the present setting of hyperfields.  

There is also a close connection between the tropical hyperfield ${\mathbb T}$ and the ``supertropical semiring'' of Izhakian--Rowen \cite{IRsupertropicalalgebra, IRmatrixalgebra}; roughly speaking, the map sending a ghost element of the supertropical semiring to the set of all tangible elements less than or equal to it identifies the two structures.

\subsection{A note on previous arXiv versions}
\label{sec:meaculpa}

This paper is a generalization to tracts of \cite{BakerBowlerHyperfield}, which is written in the more restrictive context of hyperfields.
In arXiv versions 1 through 3 of \cite{BakerBowlerHyperfield} (in which the first author was the sole author), there is a serious error which is related to the gap in \cite{AndersonDelucchi} mentioned above.  The second author noticed this mistake and found the counterexample discussed in \S\ref{sec:counterexample} below.  This made it clear that there are in fact at least two distinct notions of matroids over hyperfields (which we call ``weak'' and ``strong''), each of which admits a number of cryptomorphically equivalent axiomatizations.  The present version of the paper is our attempt to  correctly paint the landscape of matroids over hyperfields, as well as the corresponding generalization to tracts.

The problem with the previous versions of the present work occurs in the proof of Theorem 6.19 on page 29 of arXiv version 3. Shortly before the end of the proof, one finds the equation 
\[
X(e) \odot Y(e) = - X'(e) \odot Y(e) = \bigboxplus_{g \neq e} X'(g) \odot Y(g).
\] 
However, the term on the right is a set rather than a single element\footnote{When $|\underline{X} \cap \underline{Y}| \leq 3$, the proof of Theorem 6.19 goes through because in that case the hypersum $\bigboxplus_{g \neq e} X'(g) \odot Y(g)$ is single-valued (as there is just one element other than $e$ in $\underline{X}' \cap \underline{Y}$).} so the second equality sign should be $\in$ rather than $=$.
Unfortunately, this containment is not sufficient to give the desired result; indeed, the ``desired result'' is false as shown in \S\ref{sec:counterexample} below.

\subsection{Structure of the paper}

% {\bf Matt: Insert a new summary here.}
In Section~\ref{sec:examplestracts} we explain the algebraic structures which give the main motivation for tracts.
In Section~\ref{sec:MatroidsOverTracts} we present different ``cryptomorphic'' axiom systems for strong and weak matroids over tracts, 
and state the main results of duality theory.  
We also discuss (in Section~\ref{sec:counterexample}) some examples of weak $F$-matroids which are not strong, and (in Section~\ref{sec:functoriality}) push-forward operations on $F$-matroids.
%  via tract homomorphisms.
We conclude the section by showing that weak and strong $F$-matroids coincide over perfect tracts, and that doubly distributive partial hyperfields are perfect.
Proofs of the main theorems are deferred to Section~\ref{sec:proofsection}.
There are two brief Appendices at the end of the paper: in Appendix~\ref{sec:errata} we collect some errata from \cite{AndersonDelucchi}, and in 
Appendix~\ref{sec:Lorscheid} we present a simplified point of view on fuzzy rings written by Oliver Lorscheid.

\subsection{Acknowledgments}
The first author would like to thank Felipe Rincon, Eric Katz, Oliver Lorscheid, and Ravi Vakil for useful conversations.
He also thanks Dustin Cartwright, Alex Fink, Felipe Rincon, and an anonymous referee for pointing out some minor mistakes in the first arXiv version of this paper. Finally, he thanks Sam Payne and Rudi Pendavingh and two anonymous referees for helpful comments, and Louis Rowen for explaining the connection to his work with Izhakian and Knebusch. 

We are also grateful to Masahiko Yoshinaga for pointing out a problem with an earlier version of Remark~\ref{rmk:inducedhyperaddition}, to Ting Su for suggesting improvements to the proof of Theorem \ref{thm:Prop5.6} and to Daniel Wei\ss auer for finding the counterexample given as Example~\ref{eg:danscex}. 
{ We are especially grateful to Laura Anderson for her detailed feedback on all the various drafts of this paper, and to Ting Su and an anonymous referee for additional corrections and suggestions.
We also thank Oliver Lorscheid for contributing Appendix~\ref{sec:Lorscheid}.}
% I thank Felipe Rincon and Eric Katz for pointing out the key differences between valuated and oriented matroids,
% Ravi Vakil and Oliver Lorscheid for useful conversations on hyperstructures, and Laura Anderson and Emanuele Delucchi for writing their paper,
% the discovery of which saved me a lot of work.
% Thanks also to Eric Katz and Robin Thomas for pointing me toward the papers \cite{PvZLifts,PvZSkew}, and to the anonymous referees for their helpful feedback.
% Finally, I'd like to thank Laura Anderson, Dustin Cartwright, Alex Fink, and Felipe Rincon for pointing out mistakes in the first arXiv version of this paper, Sam Payne and Rudi Pendavingh for helpful comments, and Louis Rowen for explaining the connection to his work with Izhakian and Knebusch.
% In the next arXiv version thank Ben Steinberg as well!

\section{Examples of tracts}\label{sec:examplestracts}

In this section, we will explain some of the motivating examples of tracts. The tract axioms (T0)-(T3) are motivated by the fact that they appear to be precisely the properties needed in order to establish the basic cryptomorphisms of matroid theory.
Note, however, that many of the tracts in this section satisfy somewhat stronger properties.  For example, $N_G$ is frequenty an {\bf ideal} in $\NN[G]$, closed under addition (and therefore, by (T3), under multiplication by arbitrary elements of $\NN[G]$). 
% and not just by elements of $G$. 
Our first example lacks these nice properties, and illustrates the freedom allowed by our definition.

\begin{example}
The {\bf initial tract} ${\mathbb I}$ is defined to be $(G=\{-1,1\}, N_G=\{0, 1 + (-1)\})$, with the multiplication on $G$ being the usual one. Our terminology arises from the fact that ${\mathbb I}$ is the initial object in the category whose objects are tracts and whose maps are tract homomorphisms.
\end{example}

\subsection{Hyperrings and hyperfields} \label{sec:hyperfields}

 A hypergroup (resp. hyperring, hyperfield) is an algebraic structure similar to a group (resp. ring, field) except that addition is multivalued.  
More precisely, addition in a hypergroup is a {\bf hyperoperation} on a set $S$, i.e., a map $\boxplus$ from $S \times S$ to the collection of non-empty subsets of $S$.
All hyperoperations in this paper will be {\em commutative}, though the non-commutative case is certainly interesting as well.
% All hypergroups and hyperrings in this paper will be {\it commutative}.  

%(For more on hyperstructures, see for example \cite{ConnesConsani} and \cite[Appendix B]{JunValuations}.)

If $A,B$ are non-empty subsets of $S$, we define
\[
A \boxplus B := \bigcup_{a \in A, b \in B} (a \boxplus b)
\]
and we say that $\boxplus$ is {\bf associative} if $a \boxplus (b \boxplus c) = (a \boxplus b) \boxplus c$ for all $a,b,c \in S$.

Given an associative hyperoperation $\boxplus$, we define the hypersum $x_1 \boxplus \cdots \boxplus x_m$ of $x_1,\ldots,x_m$ for $m \geq 2$ recursively by the formula
\[
x_1 \boxplus \cdots \boxplus x_m := \bigcup_{x' \in x_2 \boxplus \cdots \boxplus x_m} x_1 \boxplus x'.
\]

\begin{defn} \label{def:hypergroup}
A {\bf hypergroup} is a tuple $(G,\boxplus,0)$, where $\boxplus$ is an associative hyperoperation on $G$ such that:
\begin{itemize}
\item (H0) $0 \boxplus x = \{ x \}$ for all $x \in G$.
\item (H1) For every $x \in G$ there is a unique element of $G$ (denoted $-x$ and called the {\bf hyperinverse} of $x$) such that $0 \in x\boxplus -x$.
\item (H2) $x \in y \boxplus z$ if and only if $z \in x \boxplus (-y)$.
\end{itemize}
\end{defn}

\begin{remark}
Axiom (H2) is called {\em reversibility}, and in the literature a hypergroup is often only required to satisfy (H0) and (H1); a hypergroup satisfying (H2) is called a {\em canonical hypergroup}.  Since we will deal only with hypergroups satisfying (H2), we will drop the (old-fashioned sounding) adjective `canonical'.
\end{remark}

% The proof of the following is immediate:

% \begin{lemma} \label{lem:reversibility}
% If $G$ is a hypergroup and $x,y,z \in G$, then $0 \in x \boxplus y \boxplus z$ if and only if $-z \in x \boxplus y$.
% \end{lemma}

\begin{defn} \label{def:hyperring}
A {\bf hyperring} is a tuple $(R,\odot,\boxplus,1,0)$ such that:
\begin{itemize}
\item $(R,\odot,1)$ is a commutative monoid.
\item $(R,\boxplus,0)$ is a a commutative hypergroup.
\item (Absorption rule) $0 \odot x = x \odot 0 = 0$ for all $x \in R$.
\item (Distributive Law) $a \odot (x \boxplus y) = (a \odot x) \boxplus (a \odot y)$ for all $a,x,y \in R$, and similarly for right-multiplication.
\end{itemize}
\end{defn}

As usual, we will denote a hyperring by its underlying set $R$ when no confusion will arise.
Note that any unital ring $R$ may be considered in a trivial way as a hyperring.
We will often write $xy$ (resp. $x/y$) instead of $x \odot y$ (resp. $y^{-1} \odot x$) if there is no risk of confusion.

\begin{remark}
Our notion of hyperring is sometimes called a {\em Krasner hyperring} in the literature; it is a special case of a more general class of algebraic structures in which one allows multiplication to be multivalued as well.  
Since we will not make use of more general hyperrings in this paper, and since (following \cite{ConnesConsani}) we will use the term `Krasner hyperfield' for something different (see Example~\ref{ex:KrasnerHyperfield} below), we will not use the term `Krasner hyperring'.
\end{remark}

\begin{remark}
If we just require $(R,\boxplus,0)$ in Definition~\ref{def:hyperring} to satisfy (H0) and (H1), it follows automatically from the distributive law that it also satisfies (H2).
\end{remark}

\begin{remark} \label{rmk:inducedhyperaddition}
If $R$ is a commutative ring with $1$ and $G$ is a subgroup of the group $R^\times$ of units in $R$, then the set $R/G$ of orbits for the action of $G$ on $R$ by multiplication has a natural hyperring structure (cf.~\cite[Proposition 2.5]{ConnesConsani}), given by taking an orbit to be in the hypersum of two others if it is a subset of their setwise sum.
\end{remark}

\begin{defn}
A hyperring $F$ is called a {\bf hyperfield} if $0 \neq 1$ and every non-zero element of $F$ has a multiplicative inverse.  
\end{defn}

\subsection{Examples}

We now give some examples of hyperfields which will be important to us in the sequel.

\begin{example}
(Fields) If $F=K$ is a field, then $F$ can be trivially considered as a hyperfield by setting $a \odot b = a \cdot b$ and $a \boxplus b = \{ a+b \}$.
\end{example}

\begin{example} \label{ex:KrasnerHyperfield}
(Krasner hyperfield) Let ${\mathbb K} =  \{ 0,1 \}$ with the usual multiplication rule, but with hyperaddition defined by
$0\boxplus x=x\boxplus 0=\{x \}$ for $x=0,1$ and $1\boxplus 1 = \{ 0,1 \}$.  Then ${\mathbb K}$ is a hyperfield, called the {\bf Krasner hyperfield} by Connes and Consani in \cite{ConnesConsani}.
This is the hyperfield structure on $\{ 0, 1 \}$ induced (in the sense of Remark~\ref{rmk:inducedhyperaddition}) by the field structure on $F$, for any field $F$, with respect to the trivial valuation $v : F \to \{ 0,1 \}$ sending $0$ to $0$ and all non-zero elements to $1$.
\end{example}

\begin{example}
(Tropical hyperfield) Let ${\mathbb T}_+ := {\mathbb R} \cup \{ -\infty \}$, and for $a,b \in {\mathbb T}_+$ define $a\cdot b = a+b$ (with $-\infty$ as an absorbing element).  
The hyperaddition law is defined by setting $a \boxplus b = \{ \max (a,b) \}$ if $a \neq b$ and $a \boxplus b = \{ c \in {\mathbb T}_+ \; | \; c \leq a \}$ if $a = b$.  (Here we use the standard total order on ${\mathbb R}$ and set $-\infty \leq x$ for all $x \in {\mathbb R}$.)  Then ${\mathbb T}_+$ is a hyperfield, called the {\bf tropical hyperfield}.
The additive hyperidentity is $-\infty$ and the multiplicative identity is $0$.
Because it can be confusing that $0,1 \in {\mathbb R}$ are not the additive (resp. multiplicative) identity elements in ${\mathbb T}_+$, we will work instead with the isomorphic hyperfield ${\mathbb T} := {\mathbb R}_{\geq 0}$ in which $0,1 \in {\mathbb R}$ are the additive (resp. multiplicative) identity elements and multiplication is the
usual multiplication.  Hyperaddition is defined so that the map ${\rm exp}: {\mathbb T}_+ \to {\mathbb T}$ is an isomorphism of hyperfields.
\end{example}

\begin{example}
(Valuative hyperfields) More generally, if $\Gamma$ is any totally ordered abelian group (written multiplicatively), there is a canonical hyperfield structure on $\Gamma \cup \{ 0 \}$ defined in a similar way as for ${\mathbb T}$.
The hyperfield structure on $\Gamma \cup \{ 0 \}$ is induced from that on $F$ by $\| \cdot \|$ for any surjective norm $\| \cdot \| : F \onto \Gamma \cup \{ 0 \}$ on a field $F$.
We call a hyperfield which arises in this way a {\bf valuative hyperfield}.  In particular, both ${\mathbb K}$ and ${\mathbb T}$ are valuative hyperfields.
\end{example}

\begin{example}
(Hyperfield of signs) Let ${\mathbb S} := \{ 0, 1, -1 \}$ with the usual multiplication law, and hyperaddition defined by $1 \boxplus 1 = \{ 1 \}$, $-1 \boxplus -1 = \{ -1 \}$, $x \boxplus 0 = 0 \boxplus x = \{ x \}$, and $1 \boxplus -1 = -1 \boxplus 1 = \{ 0, 1, -1 \}$.  Then ${\mathbb S}$ is a hyperfield, called the {\bf hyperfield of signs}. 
% The underlying multiplicative monoid of ${\mathbb S}$ is sometimes denoted by ${\mathbb F}_{1^2}$.  
The hyperfield structure on $\{ 0,1,-1 \}$ is induced from that on ${\mathbb R}$ by the map $\sigma : {\mathbb R} \to \{ 0, 1, -1 \}$ taking $0$ to $0$ and a nonzero real number to its sign.  
\end{example}

\begin{example}
(Weak hyperfields and the weak hyperfield of signs) 
For any abelian group $G$ and any self-inverse element $\epsilon$ of $G$, there is a hyperfield $W(G, \epsilon)$ given as follows: the underlying set is $G \cup \{0\}$, the multiplication is given by that of $G$ together with the rule $0 \cdot x = 0$, and the hyperaddition is given by $0 \boxplus x = \{x\}$, $x \boxplus (\epsilon \cdot x) = G \cup \{0\}$, and $x \boxplus y = G$ for any nonzero $x$ and $y$ with $y \neq \epsilon \cdot x$. It is easy to check that this really does give a hyperfield; for example both sides of the equation for associativity evaluate to $G \cup \{0\}$ if all summands are nonzero. We shall call such hyperfields {\em weak hyperfields}.

A particularly important example is the {\em weak hyperfield of signs} ${\mathbb W} = W(\{ 1,-1 \}, -1)$.
The underlying multiplicative monoid of ${\mathbb W}$ is the same as for ${\mathbb S}$.
The hyperfield structure on $\{ 0,1,-1 \}$ is induced from that on ${\mathbb F}_p$ by the map $\sigma : {\mathbb F}_p \to \{ 0, 1, -1 \}$ taking $0$ to $0$, all squares to 1 and all nonsquares to $-1$ for any prime number $p>3$ congruent to 3 modulo 4.  
\end{example}

\begin{example}
(Phase hyperfield) Let ${\mathbb P} := S^1 \cup \{ 0 \}$, where $S^1 = \{ z \in {\mathbb C} \; | \; |z|=1 \}$ is the complex unit circle.  Multiplication is defined as usual, and the hyperaddition law is defined for $x,y \neq 0$ by setting $x \boxplus -x := \{ 0, x, -x \}$ and 
$x \boxplus y := \{ \frac{\alpha x + \beta y}{\| \alpha x + \beta y \|} \; | \; \alpha, \beta \in {\mathbb R}_{>0} \}$ otherwise.
The hyperfield structure on $S^1 \cup \{ 0 \}$ is induced from that on ${\mathbb C}$ by the map $p : {\mathbb R} \to S^1 \cup \{ 0 \}$ taking $0$ to $0$ and a nonzero complex number $z$ to its phase $z / |z| \in S^1$.
\end{example}

Many other interesting examples of hyperstructures are given in Viro's papers  \cite{ViroDequant,Viro} and the papers \cite{ConnesConsaniAbsolute,ConnesConsani} of Connes and Consani.  Here are a couple of examples taken from these papers:

\begin{example}
\label{ex:triangle}
(Triangle hyperfield) Let ${\mathbb V}$ be the set ${\mathbb R}_{\geq 0}$ of nonnegative real numbers with the usual multiplication and the hyperaddition rule
\[
a \boxplus b := \{ c \in {\mathbb R}_{\geq 0} \; : \; |a-b| \leq c \leq a+b \}.
\]
(In other words, $a \boxplus b$ is the set of all real numbers $c$ such that there exists a Euclidean triangle with side  lengths $a, b, c$.)
Then ${\mathbb V}$ is a hyperfield, closely related to the notion of {\em Litvinov-Maslov dequantization} (cf.~\cite[\S{9}]{ViroDequant}).
\end{example}

\begin{example}
(Ad{\`e}le class hyperring) If $K$ is a global field and $A_K$ is its ring of ad{\`e}les, the commutative monoid $A_K / K^\times$ (which plays an important role in Connes' conjectural approach to proving the Riemann hypothesis) is naturally endowed with the structure of a hyperring by Remark~\ref{rmk:inducedhyperaddition}.  It is, moreover, an algebra over the Krasner hyperfield ${\mathbb K}$ in a natural way.
One of the interesting discoveries of Connes and Consani \cite{ConnesConsani} is that if $K$ is the function field of a curve $C$ over a finite field, 
the groupoid of prime elements of the hyperring $A_K / K^\times$ is canonically isomorphic to the loop groupoid of the maximal abelian cover of $C$.
\end{example}

\begin{remark}
There are examples of hyperfields which do not arise from the construction given in Remark~\ref{rmk:inducedhyperaddition}; see \cite{Massouros}.
\end{remark}

\subsection{The tract associated to a hyperfield}

A fundamental example of a tract is the tract associated to a hyperfield $K$, where we set $G = K \backslash \{ 0 \}$ and a formal sum 
$\sum_i a_i g_i \in \NN[G]$ with $a_i \in \NN$ and $g_i \in G$ belongs to $N_G$ if and only if $0 \in \boxplus_i a_i g_i$ in $K$.
% (See \S\ref{sec:hyperfields} for an overview of the theory of hyperfields.)

With our general definition of matroids over tracts,\footnote{By a matroid over a hyperfield, we mean a matroid over the corresponding tract, and similarly for partial fields in the sense of \S\ref{sec:partialfields} below.}
we will find for example that:

\begin{itemize}
\item A (strong or weak) matroid over ${\mathbb S}$ is the same thing as an {\bf oriented matroid} in the sense of Bland--Las Vergnas \cite{BlandLasVergnas}.
\item A (strong or weak) matroid over ${\mathbb T}$ is the same thing as a {\bf valuated matroid} in the sense of Dress--Wenzel \cite{DressWenzelVM}.
\item There exists a weak matroid over ${\mathbb V}$ which is not a strong matroid.
\end{itemize}

Anderson and Delucchi consider aspects of both weak and strong matroids over ${\mathbb P}$ in \cite{AndersonDelucchi}, but there is a mistake in their proof that the circuit, Grassmann--Pl{\"u}cker, and dual pair axioms for phased matroids are all equivalent (cf.~Appendix~\ref{sec:errata}). A counterexample due to Daniel Wei\ss auer shows that weak ${\mathbb P}$-matroids are not the same thing as strong ${\mathbb P}$-matroids (see Example~\ref{eg:danscex}).

% Thus the notion of matroids over a hyperfield is general enough to include not only classical linear subspaces and matroids in the usual sense, but also the three different flavors of enhanced matroids above.
% What is particularly noteworthy is that matroids over hyperfields 
% are also sufficiently specific that one can prove a number of non-trivial theorems about them; for example, they 

Both weak and strong matroids over tracts admit a duality theory which generalizes the existing duality theories in each of the above examples.  All known proofs of the basic duality theorems for oriented or valuated matroids are rather long and involved. 
%(not to mention tricky).  
One of our goals is to give a unified treatment of such duality results so that one only has to do the hard work once.

\subsection{Homomorphisms of hyperfields}

Our definition of homomorphisms of tracts is compatible with the usual definition of hyperfield homomorphisms with respect to the realization of hyperfields as tracts.
In order to make this precise, we recall the following:

\begin{defn}
A {\bf hypergroup homomorphism} is a map $f : G \to H$ such that $f(0)=0$ and $f(x \boxplus y) \subseteq f(x) \boxplus f(y)$ for all $x,y \in G$.

A {\bf hyperring homomorphism} is a map $f : R \to S$ which is a homomorphism of additive hypergroups as well as a homomorphism of multiplicative monoids (i.e.,  $f(1)=1$ and 
$ f(x \odot y)=f(x) \odot f(y)$ for $x,y \in R$).

A {\bf hyperfield homomorphism} is a homomorphism of the underlying hyperrings.  
\end{defn}

With these definitions and the construction of the tract associated to a hyperfield, it is not hard to see that the category of hyperfields is a full subcategory of
the category of tracts.  The main observation needed is the following lemma:

\begin{lemma}
If $F,F'$ are hyperfields and $f:F \to F'$ is a homomorphism of tracts, then for $x,y \in F$ we have $f(-x) = -f(x)$ and $f(x\boxplus y) \subseteq f(x) \boxplus f(y)$.
\end{lemma}

\begin{proof}
Let $G = F^\times$ and $G' = (F')^\times$.  If $x=0$ or $y=0$ the result is trivial, so we may assume that $x,y \in G$.
Since $-x + x \in N_G$ and $f$ is a homomorphism of tracts, $f(-x) + f(x) \in N_{G'}$, which by Lemma~\ref{lem:negatives} implies that $f(-x)=-f(x)$.
Similarly, if $z \in x \boxplus y$ then $0 \in -z \boxplus x \boxplus y$, which means that $-z + x + y \in N_G$.  
Thus $-f(z) + f(x) + f(y) \in N_{G'}$, which implies that $0 \in -f(z) \boxplus f(x) \boxplus f(y)$ and thus $f(z) \in f(x) \boxplus f(y)$.
\end{proof}

% For the following examples, 
% We define the {\bf kernel} ${\rm ker}(f)$ of a hyperring homomorphism $f$ to be $f^{-1}(0)$. 
% Note that a hyperring homomorphism must send units to units, and therefore if $f : R \to S$ is a homomorphism and $R$ is a hyperfield, we must have ${\rm ker}(f)=\{ 0 \}$.

% \begin{remark}
% When the hyperring structure on a demigroup $A$ is induced from a map $f : R \twoheadrightarrow A$ as in Remark~\ref{rmk:inducedhyperaddition}, the map $f$ becomes a homomorphism of hyperrings with ${\rm ker}(f) = \{ 0 \}$.
% \end{remark}

\begin{example}
A hyperring homomorphism from a commutative ring $R$ with $1$ to the Krasner hyperfield ${\mathbb K}$ (cf.~Example~\ref{ex:KrasnerHyperfield}) is the same thing as a prime ideal of $R$, via the correspondence ${\mathfrak p} := f^{-1}(0)$.
%  {\rm ker}(f)$. 
\end{example}

\begin{example}
A hyperring homomorphism from a commutative ring $R$ with $1$ to the tropical hyperfield ${\mathbb T}$ is the same thing as a prime ideal ${\mathfrak p}$ of $R$ together with a real valuation on the residue field of ${\mathfrak p}$ (i.e., the fraction field  of $R/{\mathfrak p}$).  
 (Similarly, a hyperring homomorphism from $R$ to $\Gamma \cup \{ 0 \}$ for some totally ordered abelian group $\Gamma$
is the same thing as a prime ideal ${\mathfrak p}$ of $R$ together with a Krull valuation on the residue field of ${\mathfrak p}$.)
In particular, a hyperring homomorphism from a field $K$ to ${\mathbb T}$ is the same thing as a real valuation on $K$.
These observations allow one to reformulate the basic definitions in Berkovich's theory of analytic spaces \cite{BerkovichBook} in terms of hyperrings, though we will not explore this further in the present paper.
\end{example}

\begin{example}
 A hyperring homomorphism from a commutative ring $R$ with $1$ to the hyperfield of signs ${\mathbb S}$ is the same thing as a prime ideal ${\mathfrak p}$ together with an ordering on the residue field of ${\mathfrak p}$ in the sense of ordered field theory (see e.g. \cite[\S{3}]{Marshall}).
In particular, a hyperring homomorphism from a field $K$ to ${\mathbb S}$ is the same thing as an ordering on $K$.
This observation allows one to reformulate the notion of {\em real spectrum} \cite{BasuPollackRoy,MarshallBook} in terms of hyperrings, and provides an interesting lens through which to view the analogy between Berkovich spaces and real spectra.
\end{example}

\subsection{Partial fields}  \label{sec:partialfields}

The following definition is taken from \cite[Definitions 2.1 and 3.1]{PvZSkew}:

\begin{defn}
A {\bf partial field} $P$ is a pair $(G,R)$ consisting of a commutative ring $R$ with $1$ and a subgroup $G$ of the group of units of $R$ such that:
\begin{itemize}
\item[(PF1)] $-1$ belongs to $G$.
\item[(PF2)] $G$ generates the ring $R$.
\end{itemize}
\end{defn}

Note that some authors omit axiom (PF2) and instead consider pairs $(F,R)$ as above modulo a certain equivalence relation which yield the same
objects.

Note also that a partial field with $G = R \backslash \{ 0 \}$ is the same thing as a {\em field}.

\begin{example}
(Partial fields) There are many interesting examples of partial fields given in \cite{PvZLifts}.  We mention for example the following:
\begin{itemize}
\item The {\bf regular partial field} ${\mathbb U}_0 := (\{\pm 1 \}, {\mathbb Z})$.
\item The {\bf dyadic partial field} ${\mathbb D} := (\langle -1,2 \rangle, {\mathbb Z}[\frac{1}{2}])$.
% \item The {\bf near-regular partial field} ${\mathbb U}_1 := (\{ \langle -1, T, 1-T \rangle \} \cup \{ 0 \}, {\mathbb Z}[\frac{1}{T},\frac{1}{1-T}])$, where $T$ is an indeterminate.
\end{itemize}
\end{example}

There are numerous classical theorems about representability of matroids which can be interpreted and/or enriched using the language of partial fields.  For example:

\begin{example} \label{ex:regular}
A matroid is called {\bf regular} if it is representable over every field.  By \cite{Tutte} (see also \cite[Theorem 2.29]{PvZLifts}) 
% (which generalizes classical result of Tutte), 
the following are equivalent:
\begin{enumerate}
\item $M$ is regular.
\item $M$ is representable over every partial field.
\item $M$ is representable over ${\rm GF}(2)$ and ${\rm GF}(3)$.
\item $M$ is representable over the partial field ${\mathbb U}_0$.
\end{enumerate}
\end{example}

\begin{example} \label{ex:dyadic}
A matroid is called {\bf dyadic} if it is representable over every field of characteristic different from $2$.  By \cite{Whittle} (see also \cite[Theorem 4.3]{PvZLifts}), the following are equivalent:
\begin{enumerate}
\item $M$ is dyadic.
\item $M$ is representable over ${\rm GF}(3)$ and ${\rm GF}(5)$.
\item $M$ is representable over the partial field ${\mathbb D}$.
\end{enumerate}
\end{example}

% \begin{example} \label{ex:nearregular}
% A matroid is called {\bf near-regular} if it is representable over every field except possibly ${\rm GF(2)}$.  By \cite{Whittle} (see also \cite[Theorem 4.5]{PvZLifts}), the following are equivalent:
% \begin{enumerate}
% \item $M$ is near-regular.
% \item $M$ is representable over ${\rm GF}(3), {\rm GF}(4)$, and ${\rm GF}(5)$.
% \item $M$ is representable over the partial field ${\mathbb U}_1$.
% \end{enumerate}
% \end{example}

% Some interesting examples of matroids representable over {\em skew} partial fields are given in \cite[\S{3.5}]{PvZSkew}.
% For example, the famous {\em non-Pappus matroid} is representable over a skew field but not over any field. 
% And the direct sum of the ternary Reid geometry $R_9$ and the rank 3 Dowling geometry over the quaternion group is representable over a skew
% partial field but not over any partial field.  (In fact, there is such a representation over a {\em finite} skew partial field; the analogue of Wedderburn's
% theorem fails for skew partial fields.)

% \medskip

% These examples show that matroids over skew partial fields are indeed of geometric interest, which is one reason we allow our tracts in this paper to be 
% non-commutative.\footnote{The other reason, of course, is that our proofs do not require commutativity!}

% \subsection{The tract associated to a (skew) partial field or partial hyperfield} \label{sec:partialhyperfieldtract}
% Another basic example of a tract is the tract $F$ associated to a skew partial field $P=(G,R)$, where

\medskip

We can associate a tract to a partial field $P=(G,R)$ by declaring that
a formal sum $\sum a_i g_i \in \NN[G]$ belongs to $N_G$ if and only if $\sum a_i g_i = 0$ in $R$.
% (See \S\ref{sec:partialfields} for an overview of the theory of partial fields.)

\medskip

Our definition of matroid over a partial field\footnote{As before, by a matroid over a partial field $F$ we mean a matroid over the corresponding tract.} $P$ will have the property that (weak or strong) $P$-matroids are the same thing as matroids representable over $P$ in the sense of \cite{PvZLifts}.
% \footnote{Although we suspect that a matroid over a {\em skew} partial field $P$ is the same thing as a matroid representable over $P$ in the sense of \cite{PvZSkew}, we have not checked this.}
% \footnote{{\bf Matt:} We should look into the relation between representability over a skew partial field $P$ in the sense of \cite{PvZSkew} and matroids over $P$.}  
In particular, a regular (resp. dyadic)
% ,near-regular) 
matroid is the same thing as a (weak or strong) matroid over the partial field ${\mathbb U}_0$ (resp. ${\mathbb D}$).

\subsection{Partial hyperfields}
\label{sec:partialhyperfields}

We define a {\bf partial hyperfield} to be a pair $(G,R)$, where $G$ is a subgroup of the group of units of a (commutative) hyperring $R$ which
is an integral domain, i.e., $xy=0$ in $R$ implies that $x=0$ or $y=0$.  
Partial hyperfields generalize both hyperfields and partial fields in a natural way.
We will set $P = G \cup \{ 0 \}$ and denote the partial hyperfield $(G,R)$ simply by $P$ when no confusion is likely to arise.

We can associate a tract to a partial hyperfield by declaring that a formal sum $\sum a_i g_i \in \NN[G]$ belongs to $N_G$ if and only if $0 \in \boxplus a_i g_i$.

We will see in \S\ref{sec:perfect} below that if $P$ is a {\em doubly distributive} partial hyperfield, every weak matroid over $P$ is automatically strong.

% \subsection{Axiomatic partial fields}

% Insert a definition and brief discussion.\footnote{{\bf Nathan} I don't think the definition is very illuminating, so we should omit it unless we want to do something with it.}

\subsection{Fuzzy rings} \label{sec:fuzzyringtract}

% We are grateful to Oliver Lorscheid for allowing us to share some of his insights in this section.\footnote{{\bf Matt:} Either write out detailed proofs, with Oliver's
% permission, or ask him to write a short appendix?  I'm basing the statements here on Oliver's note ``FROM HYPERFIELDS TO FUZZY RINGS TO ORDERED BLUEPRINTS... ...AND BACK'' which he sent me by email.}
% We refer the reader to \cite{GiansiracusaJunLorscheid} and the papers of Dress and Dress--Wenzel \cite{Dress,DressWenzelGP} for the precise definition of fuzzy rings
% (which we will not actually use).

A {\bf fuzzy ring} in the sense of Dress--Wenzel (see, e.g., \cite{DressWenzelPM}) is a tuple $(K; +; \cdot; \epsilon; K_0)$ where $K$ is a set, 
$+$ and $\cdot$ are binary operations on $K$, $\epsilon \in K$, and $K_0 \subseteq K$ satisfying the following axioms:

\begin{itemize}
\item[(FR0)] $(K,+)$ and $(K,\cdot)$ are abelian semigroups with neutral elements $0,1$, respectively.
\item[(FR1)] $0 \cdot x = 0$ for all $x \in K$.
\item[(FR2)] If $x,y \in K$ and $\alpha \in K^* := \{ \beta \in K \; : \; 1 \in \beta \cdot K \}$ is a unit in $K$, then $\alpha \cdot (x+y) = \alpha \cdot x + 
\alpha \cdot y$.
\item[(FR3)] $\epsilon^2 = 1$.
\item[(FR4)] $K_0$ is a proper semiring ideal, i.e., $K_0 + K_0 \subseteq K_0$, $K \cdot K_0 \subseteq K_0$, $0 \in K_0$, and $1 \not\in K_0$.
\item[(FR5)] For $\alpha \in K^*$ we have $1 + \alpha \in K_0$ if and only if $\alpha = \epsilon$.
\item[(FR6)] If $x_1,x_2,y_1,y_2 \in K$ and $x_1 + y_1, x_2 + y_2 \in K_0$ then $x_1 \cdot x_2 + \epsilon \cdot y_1 \cdot y_2 \in K_0$.
\item[(FR7)] If $x,y,z_1,z_2 \in K$ and $x + y \cdot (z_1 + z_2) \in K_0$ then $x + y \cdot z_1 + y \cdot z_2 \in K_0$.
\end{itemize}

By an observation of Lorscheid (cf. Appendix~\ref{sec:Lorscheid}), the category of fuzzy rings together with weak homomorphisms between them is equivalent to the category whose objects are quintuples $(K;+;\cdot;\epsilon;K_0)$ for which $(K,+,\cdot)$ is a commutative semiring equal to $\N[K^\ast]$ and such that $\epsilon\in K^\ast$ and $K_0\subseteq K$ satisfy (FR4), (FR5) and (FR6).  Such quintiples are special cases of tracts (with $G=K^\ast$).

Using Lorscheid's observation,
% we see that every fuzzy ring is weakly isomorphic to a tract, and any tract satisfying (F1) and (F2) is a fuzzy ring.
fuzzy rings with weak homomorphisms between them can be viewed as a full subcategory of the category of tracts.
Moreover, it follows from the Grassmann-Pl{\"u}cker characterization in \cite{DressWenzelGP} that a matroid over a fuzzy ring $F$ in the sense of Dress--Wenzel is the same thing as a strong matroid in our sense over the corresponding tract.
Therefore our theory generalizes that of Dress and Wenzel.

\section{Matroids over tracts}
\label{sec:MatroidsOverTracts}

Let $E$ be a finite set.
In this section, we will define what it means to be a strong (resp. weak) {\bf matroid on $E$ with coefficients in a tract $F$}, or (for brevity) a strong (resp. weak) {\bf matroid over $F$} or {\bf $F$-matroid}.
Our definition will be such that:

\begin{itemize}
\item When $F=K$ is a field, a strong or weak matroid on $E$ with coefficients in $K$ is the same thing as a vector subspace of $K^E$ in the usual sense.
\item A strong or weak matroid over ${\mathbb K}$ is the same thing as a {\bf matroid}.
\item A strong or weak matroid over ${\mathbb T}$ is the same thing as a {\bf valuated matroid} in the sense of Dress--Wenzel \cite{DressWenzelVM}.
\item A strong or weak matroid over ${\mathbb S}$ is the same thing as an {\bf oriented matroid} in the sense of Bland--Las Vergnas \cite{BlandLasVergnas}.
% \item A matroid over ${\mathbb P}$ is the same thing as an {\bf phased matroid} in the sense of Anderson--Delucchi \cite{AndersonDelucchi}.
\item A strong or weak matroid over the regular partial field ${\mathbb U}_0$ is the same thing as a regular matroid.
\end{itemize}

See \S\ref{sec:whythesame} for further details on the compatibility of our notion of $F$-matroid with various existing definitions in these particular examples.

\subsection{Linear independence, spans, and orthogonality} \label{sec:modules}

If $F$ is a tract and $E$ is a set, we denote by $F^E$ the set of functions from $E$ to $F$, which carries a natural action of $F$ by pointwise multiplication.
The $F$-circuits of a (strong or weak) $F$-matroid will by definition be certain subsets of $F^E$.

\medskip

There are natural left and right actions of $F$ on $F^E$ by coordinate-wise multiplication.
If $E = \{ 1,\ldots,m \}$, we sometimes write $F^m$ instead of $F^E$.

\medskip

The {\bf support} of $X \in F^E$, denoted $\underline{X}$ or ${\rm supp}(X)$, is the set of $e \in E$ such that $X(e) \neq 0$.
If $A \subseteq F^E$, we set ${\rm supp}(A) := \{ \underline{X} \; | \; X \in A \}$.
% and we consider ${\rm supp}(A)$ as a lattice (in the poset-theoretic sense) with respect to inclusion.

\medskip

The {\bf projective space} ${\mathbb P}(F^E)$ is defined to be the set of equivalence classes of elements of $F^E$ under the equivalence relation where $X_1 \sim X_2$ if and only if $X_1 = g \odot X_2$ for some $\alpha \in G$.  
Note that the support of $X \in F^E$ depends only on its equivalence class in ${\mathbb P}(F^E)$.
We let $\pi : F^E \backslash \{ 0 \} \onto {\mathbb P}(F^E)$ denote the natural projection. 

\begin{defn} \label{defn:linindep}
(Linear independence) We say that elements $X_1,\ldots,X_k$ in $F^E$ are {\bf linearly dependent} if there exist $c_1,\ldots,c_k \in F$, not all $0$, such that
\[
c_1 X_1 + \cdots + c_k X_k \in N_G^E.
\]
Elements which are not linearly dependent are called {\bf linearly independent}.
\end{defn}

{ 
We can define linear spans in a similar way.

\begin{defn} \label{defn:linspan}
(Linear span) The {\bf linear span} of $X_1,\ldots,X_k \in F^E$ is defined to be the set
of all $X \in F^E$ such that 
\[
c_1X_1 + \cdots + c_k X_k - X \in N_G^E
\]
for some $c_1,\ldots,c_k \in F$.
\end{defn}
}

The following definitions will play an important role in the theory of duality which we develop later in this paper.

{
\begin{defn}
(Involution) Let $F$ be a tract.  An {\bf involution} of $F$ is a homomorphism $\tau : F \to F$ 
such that $\tau^2$ is the identity map.
\end{defn}

\begin{defn}
(Orthogonality) Let $F$ be a tract endowed with an involution $x \mapsto \overline{x}$, and
let $E = \{ 1,\ldots, m \}$.
The {\bf inner product} of $X=(x_1,\ldots,x_m)$ and $Y=(y_1,\ldots,y_m)$ in $F^m$ is defined to be 
\[
X \cdot Y := x_1 \involution{y}_1 + \cdots + x_m \involution{y}_m.  
\]
We say that $X,Y$ are {\bf orthogonal}, denoted $X \perp Y$, if $X \cdot Y \in N_G$.

If $S \subseteq M$, we denote by $S^\perp$ the set of all $X \in M$ such that $X \perp Y$ for all $Y \in S$.
\end{defn}

When $F$ is the field $\CC$ of complex numbers or the phase hyperfield ${\mathbb P}$, one should take the involution on $F$ to be complex conjugation. For $F \in \{ {\mathbb K}, {\mathbb T}, {\mathbb S} \}$, one should take the involution on $F$ to be the identity map.
More generally, in examples where we do not specify what the involution $x \mapsto \involution{x}$ is, the reader should take it to be the identity map.
}

\medskip

Note for later reference that for $X,Y \neq 0$, the condition $X \perp Y$ only depends on the equivalence classes of $X,Y$ in ${\mathbb P}(F^E)$.

\subsection{Modular pairs}

As in the investigation of phased matroids by Anderson--Delucchi, a key ingredient for obtaining a robust notion of matroid in the general setting of hyperfields is the concept of {\em modular pairs}. 
% We give a definition here which suffices for our purposes; for a definition in the general context of {\em lattices}, see \cite{Delucchi}.

\medskip

\begin{defn}
Let $E$ be a set and let ${\mathcal C}$ be a collection of pairwise incomparable nonempty subsets of $E$.  We say that $C_1,C_2 \in {\mathcal C}$ form a {\bf modular pair} in ${\mathcal C}$ if $C_1 \neq C_2$ and $C_1 \cup C_2$ does not properly contain a union of two distinct elements of ${\mathcal C}$.
\end{defn}

%%%%% Begin material I'm proposing to remove
{It is useful to reinterpret this definition in the language of {\em lattices}.  We recall the relevant definitions for the reader's benefit.

\medskip

Let $(S,\leq)$ be a partially ordered set (poset).  A {\bf chain} in $S$ is a totally ordered subset $J$; the {\bf length} of a chain is $\ell(J) := |J| - 1$.
The {\bf length} of $S$ is the supremum of $\ell(J)$ over all chains $J$ of $S$.
The {\bf height} of an element $X$ of $S$ is the largest $n$ such that there is a chain $X_0 < X_1 < \ldots < X_n$ in $S$ with $X_n = X$.

\medskip

Given $x \in S$ we write $S_{\leq x} = \{ y \in S \; | \; y \leq x \}$ and $S_{\geq x} = \{ y \in S \; | \; y \geq x \}$.  These are sub-posets of $S$.
Let $x,y \in S$.  If the poset $S_{\geq x} \cap S_{\geq y}$ has a unique minimal element, this element is denoted $x \vee y$ and called the {\bf join} of $x$ and $y$.
If the poset $S_{\leq x} \cap S_{\leq y}$ has a unique maximal element, this element is denoted $x \wedge y$ and called the {\bf meet} of $x$ and $y$.
The poset $S$ is called a {\bf lattice} if the meet and join are defined for any $x,y \in S$.

\medskip

Every finite lattice $L$ has a unique minimal element $0$ and a unique maximal element $1$.
An element $x \in L$ is called an {\bf atom} if $x \neq 0$ and there is no $z \in L$ with $0 < z < x$.
Two atoms $x,y \in L$ form a {\bf modular pair} if the height of $x \vee y$ is 2, i.e., $x \neq y$ and there do not exist $z,z' \in L$ with $0 < z < z' < x \vee y$.

\medskip

If ${\mathcal S}$ is any family of subsets of a set $E$, the set $U({\mathcal S}) := \{ \bigcup T \; | \; T \in {\mathcal S} \}$ forms a lattice when equipped with the partial order coming from inclusion of sets, with join corresponding to union and with the meet of $x$ and $y$ defined to be the union of all sets in ${\mathcal S}$ contained in both
$x$ and $y$.
If the elements of ${\mathcal S}$ are incomparable, then every $x \in {\mathcal S}$ is atomic as an element of $U({\mathcal S})$.
We say that two elements $x,y \in {\mathcal S}$ are a {\bf modular pair} in ${\mathcal S}$ if they are a modular pair in the lattice $U({\mathcal S})$.
}
%%%%% End material I'm proposing to remove

\medskip

Our interest in modular pairs comes in part from the observation of Anderson and Delucchi that there is a nice axiomatization of {\em phased matroids} in terms of modular pairs of {\em phased circuits}, but general pairs of phased circuits do not obey circuit elimination. 
The following facts about modular pairs will come in quite handy:

\begin{lemma} [cf.~\cite{Delucchi}] \label{lem:Delucchi}
Let ${\mathcal C}$ be a collection of non-empty incomparable subsets of a finite set $E$.  Then the following are equivalent:
\begin{enumerate}
\item ${\mathcal C}$ is the set of circuits of a matroid $M$ on $E$.
\item Every pair $C_1,C_2$ of distinct elements of ${\mathcal C}$ satisfies {\bf circuit elimination}: if $e \in C_1 \cap C_2$ then there exists $C_3 \in {\mathcal C}$ such that $C_3 \subseteq (C_1 \cup C_2) \backslash e$.
\item Every modular pair in ${\mathcal C}$ satisfies circuit elimination.
\end{enumerate}
\end{lemma}

The following lemma, which can be pieced together from \cite[Lemma 2.7.1]{WhiteCG} and \cite[Lemma 4.3]{MurotaTamura} (and also makes a nice exercise), might help the reader get a better feeling for the concept of modular pairs in the context of matroid theory:

\begin{lemma} \label{lem:white}
Let $M$ be a matroid with rank function $r$, and let $C_1, C_2$ be distinct circuits of $M$.  Then the following are equivalent:
\begin{enumerate}
\item $C_1, C_2$ are a modular pair of circuits.
\item $r(C_1 \cup C_2) + r(C_1 \cap C_2) = r(C_1) + r(C_2)$.  
\item $r(C_1 \cup C_2) = |C_1 \cup C_2| - 2$.
\item For each $e \in C_1 \cap C_2$, there is a unique circuit $C_3$ with $C_3 \subseteq (C_1 \cup C_2) \backslash e$, and this circuit has the property that $C_3$ contains the symmetric difference $C_1 \Delta C_2$.
\item { There are a basis $B$ for $M$ and a pair $e_1,e_2$ of distinct elements of $E \backslash B$ such that $C_1 = C(B,e_1)$ and $C_2 = C(B,e_2)$, where $C(B,e)$ denotes the fundamental circuit with respect to $B$ and $e$.}
\end{enumerate}
\end{lemma}

In particular, if $M$ is the cycle matroid of a connected graph $G$ then $C_1,C_2$ are a modular pair if and only if they are fundamental cycles associated to the same spanning tree $T$.

Note that for general circuits $C_1$ and $C_2$ in a matroid $M$, the {\bf submodular inequality} asserts that $r(C_1 \cup C_2) + r(C_1 \cap C_2) \leq r(C_1) + r(C_2)$.  Condition (2) of the lemma says that $C_1$ and $C_2$ form a modular pair if and only if {\em equality} holds in this inequality (hence the name ``modular pair'').

\subsection{Weak circuit axioms}

The following definition presents the first of several equivalent axiomatizations of weak matroids over tracts.
\begin{defn} \label{def:Fcircuits}
Let $E$ be a non-empty finite set and let $F = (G,N_G)$ be a tract.  
% A {\bf matroid $M$ on $E$ with coefficients in $F$}, or more simply an {\bf $F$-matroid on $E$}, is 
A subset ${\mathcal C}$ of $F^E$ is called the {\bf $F$-circuit set of a weak $F$-matroid $M$ on $E$} if ${\mathcal C}$ satisfies the following axioms:
\begin{itemize}
\item (C0) $0 \not\in {\mathcal C}$.
\item (C1) If $X \in {\mathcal C}$ and $\alpha \in F^\times$, then $\alpha \cdot X \in {\mathcal C}$.
\item (C2) [Incomparability] If $X,Y \in {\mathcal C}$ and $\underline{X} \subseteq \underline{Y}$, then there exists $\alpha \in F^\times$ such that $X = \alpha \cdot Y$.
\item ${\rm (C3)}'$ [Modular Elimination] If $X,Y \in {\mathcal C}$ are a {\bf modular pair of $F$-circuits} (meaning that $\underline{X},\underline{Y}$ are a modular pair in ${\rm supp}({\mathcal C})$) and $e \in E$ is such that $X(e)=-Y(e) \neq 0$, 
there exists an $F$-circuit $Z \in {\mathcal C}$ such that $Z(e)=0$ and $X(f) + Y(f) - Z(f) \in N_G$ for all $f \in E$.
\end{itemize} 
\end{defn}

This is equivalent to the axiom system for phased circuits given in \cite{AndersonDelucchi} in the case of phased matroids (i.e., when $F = {\mathbb P}$).  Also, the $F$-circuit $Z$ in (C3)$'$ is {\em unique}.  (Both of these observations follow easily from Lemma~\ref{lem:white}.)

\medskip

If ${\mathcal C}$ is the set of $F$-circuits of a weak $F$-matroid $M$ with ground set $E$, there is an underlying matroid (in the usual sense) $\underline{M}$ on $E$ whose circuits are the supports of the $F$-circuits of $M$.  (It is straightforward, in view of Lemma~\ref{lem:Delucchi}, to check that the circuit axioms for a matroid are indeed satisfied.)

\begin{defn} \label{def:rank}
The {\bf rank} of $M$ is defined to be the rank of the underlying matroid $\underline{M}$.
\end{defn}

\medskip

A {\bf projective $F$-circuit} of $M$ is an equivalence class of $F$-circuits of $M$ under the equivalence relation $X_1 \sim X_2$ if and only if $X_1 = g \cdot X_2$ for some $g \in F^\times$.  
Axioms (C0)-(C2) together imply that the map from projective $F$-circuits of $M$ to circuits of $\underline{M}$ which sends a projective circuit $C$ to its support is a {\em bijection}.
In particular, $M$ has only finitely many projective $F$-circuits, and one can think of a weak matroid over $F$ as a matroid $\underline{M}$ together with a function associating to each circuit $\underline{C}$ of $\underline{M}$ an element $X(\underline{C}) \in {\mathbb P}(F^E)$ such that modular elimination holds for
$\cC := \pi^{-1}(\{ X(\underline{C}) \})$.

\begin{remark}
For a version of ${\rm (C3)}'$ which holds even when $X,Y$ are not assumed to be a modular pair, see Lemma~\ref{lem:Lemma5.4}.  This weaker elimination property is not strong enough, however, to characterize weak $F$-matroids except in very special cases such as $F={\mathbb K}$.
\end{remark}

\subsection{Strong circuit axioms}

% Our second, stronger notion of representation of matroids over hyperfields can also be axiomatised in terms of circuits. The rough idea behind the axiomatisation is that, for any basis, we would like each circuit to be in some sense a linear combination of fundamental circuits with respect to that basis. Since linear combinations of elements of a hyperfield are sets rather than elements, the closest we can come to this idea is the following. Suppose that we have a basis $B$ and a circuit $C$. Let $Y$ be an element of $\Ccal$ with support $C$. For each $e \in C \setminus B$, let $X_e$ be the unique element of $\Ccal$ with support $C_e^B$ and with $X_e(e) = 1$. Then for any $f \in E$ we want $Y(f) \in \bigboxplus_{e \in C \setminus B} X(e)X_e(f)$. 

% This is clearly equivalent to the requirement that there should be elements $X$ and $\{X_e | e \in C \setminus B\}$ with the support of $X$ being $C$ and that of $X_e$ being $C_e^B$, so that for any $f \in E$ we have $0 \in X(f) \boxplus \left(\bigboxplus_e X_e(f)\right)$. For the version in the axiom, we rearrange yet further, to obtain one of the $X_e$ as a linear combination of $X$ and the remaining $X_e$. The reason for phrasing things in this way is that we do not wish to rely on the assumption that we have an underlying matroid. The formulation given below encodes a kind of circuit elimination which implies that we obtain a matroid in this way.

We say that a family of atomic elements of a lattice is {\bf modular} if the height of their join in the lattice is the same as the size of the family. If $\Ccal$ is a subset of $F^E$ then a {\bf modular family} of elements of $\Ccal$ is one such that the supports give a modular family of elements in the lattice of unions of supports of elements of $\Ccal$.

The following definition presents the first of several equivalent axiomatizations of strong matroids over tracts.
\begin{defn} \label{def:Fcircuitsprime}
% A {\bf matroid $M$ on $E$ with coefficients in $F$}, or more simply an {\bf $F$-matroid on $E$}, is 
A subset ${\mathcal C}$ of $F^E$ is called the {\bf $F$-circuit set of a strong $F$-matroid $M$ on $E$} if ${\mathcal C}$ satisfies (C0),(C1),(C2), and the following stronger version of the modular elimination axiom ${\rm (C3)}'$:

{
\begin{itemize}
% \item (C3) [Strong modular elimination] For any modular elimination structure given by $X$ and $(X_p)_{p \in P}$, there is an $F$-circuit $Z$ such that $Z(p) = 0$ for all $p \in P$ and $Z(f) \in X(f) \boxplus \left( \bigboxplus_{p \in P} X_p(f) \right)$ for any $f \in E$.
{ \item (C3) [Strong modular elimination] Suppose $X_1,\ldots,X_k$ and $X$ are $F$-circuits of $M$ which together form a modular family of size $k+1$ such that $\underline X \not \subseteq \bigcup_{1 \leq i \leq k} \underline X_i$, and for $1 \leq i \leq k$ let $$e_i \in (X \cap X_i) \setminus \bigcup_{\substack{1 \leq j \leq k \\ j \neq i}} X_j$$ be such that $X(e_i) = -X_i(e_i) \neq 0$. Then there is an $F$-circuit $Z$ such that $Z(e_i) = 0$ for $1 \leq i \leq k$ and $X_1(f) + \cdots + X_k(f) + X(f) - Z(f) \in N_G$ for every $f \in E$.}
\end{itemize}
}
\end{defn}

Any strong $F$-matroid on $E$ is in particular a weak $F$-matroid on $E$ {(take $k=1$ in the above definition)}, 
and we define the rank of such an $F$-matroid accordingly.

\medskip

{ 
Condition (C3) in Definition~\ref{def:Fcircuitsprime} may look unnatural and/or unmotivated at first glance.  However, the next result 
shows that (C3) is equivalent to a more natural-looking condition ${\rm (C3)}''$:

% {defn:linspan}
%  which says, roughly speaking, that for any basis $B$, each element of ${\mathcal C}$ is a ``linear combination'' of fundamental circuits with respect to $B$.  
% (Since linear combinations of elements of a hyperfield are sets rather than elements, one has to interpret the notion of linear combination correctly.)

\begin{theorem} \label{thm:C3primeprime}
Let ${\mathcal C}$ be a subset of $F^E$ satisfying {\rm (C0),(C1)}, and {\rm (C2)}.  Then ${\mathcal C}$ satisfies {\rm (C3)} if and only if it satisfies 
{
\begin{itemize}
% \item (C3) [Strong modular elimination] For any modular elimination structure given by $X$ and $(X_p)_{p \in P}$, there is an $F$-circuit $Z$ such that $Z(p) = 0$ for all $p \in P$ and $Z(f) \in X(f) \boxplus \left( \bigboxplus_{p \in P} X_p(f) \right)$ for any $f \in E$.
{ \item ${\rm (C3)}''$
The support of ${\mathcal C}$ is the set of circuits of a matroid $\underline{M}$, and for every $X \in {\mathcal C}$ and every basis $B$ of $\underline{M}$, $X$ is in the linear span of the vectors $X_{B,e}$ for $e \in E \setminus B$, where 
$X_{B,e}$ denotes the unique element of ${\mathcal C}$ with $X_{B,e}(e)=1$ whose support is the fundamental circuit of $e$ with respect to $B$.
}
\end{itemize}
}
\end{theorem}

\begin{remark} Condition ${\rm (C3)}''$ is equivalent to the statement that 
% there exist $c_e \in F$ such that $X(f) \in \bigboxplus_{e \in E \setminus B} c_e X_{B,e}(f)$ for all $f \in E$, since if such $c_e$ exist we must have $c_e = X(e)$ for all $e$.
the support of ${\mathcal C}$ is the set of circuits of a matroid $\underline{M}$, and for every $X \in {\mathcal C}$ and every basis $B$ of $\underline{M}$ we have
\begin{equation} \label{eq:lincomb}
X(f) - \sum_{e \in E \setminus B} X(e)X_{B,e}(f) \in N_G
\end{equation}
for all $f \in E$.
\end{remark}

Despite its naturality, condition ${\rm (C3)}''$ has the disadvantage that we need to know {\it a priori} that the support of ${\mathcal C}$ is the set of circuits of a matroid.  Another reason to prefer (C3) over ${\rm (C3)}''$ is that the former is a more direct generalization of the weak modular elimination axiom ${\rm (C3)}'$.
On the other hand, condition ${\rm (C3)}''$ has a more direct relationship to the axioms for $F$-vectors developed by Anderson in \cite{AndersonVectors}.

\medskip

We provide a proof of Theorem~\ref{thm:C3primeprime} in \S\ref{sec:C3primeprime}.
}

\subsection{Grassmann-Pl{\"u}cker functions}
\label{sec:GPsection}

We now describe a cryptomorphic characterization of weak and strong matroids over a tract $F$ in terms of {\bf Grassmann-Pl{\"u}cker functions} (called ``chirotopes'' in the theory of oriented matroids and ``phirotopes'' in \cite{AndersonDelucchi}).  In addition to being interesting in its own right, this description will be crucial for establishing a duality theory
for matroids over $F$.

\medskip

\begin{defn}
Let $E$ be a non-empty finite set, let $F=(G,N_G)$ be a tract, and let $r$ be a positive integer.  A { (strong)} {\bf Grassmann-Pl{\"u}cker function of rank $r$ on $E$ with coefficients in $F$} is a function $\varphi : E^r \to F$ such that:
\begin{itemize}
\item (GP1) $\varphi$ is not identically zero.
\item (GP2) $\varphi$ is alternating, i.e., $\varphi(x_1,\ldots,x_i, \ldots, x_j, \ldots, x_r)=-\varphi(x_1,\ldots,x_j, \ldots, x_i, \ldots, x_r)$ and $\varphi(x_1,\ldots, x_r) = 0$ if $x_i = x_j$ for some $i \neq j$.
\item (GP3) [Grassmann--Pl{\"u}cker relations] For any two subsets $\{ x_1,\ldots,x_{r+1} \}$ and $\{ y_1,\ldots,y_{r-1} \}$ of $E$,
\begin{equation}
\label{eq:GP3}
\sum_{k=1}^{r+1} (-1)^k \varphi(x_1,x_2,\ldots,\hat{x}_k,\ldots,x_{r+1}) \cdot \varphi(x_k,y_1,\ldots,y_{r-1}) \in N_G.
\end{equation}
\end{itemize}
\end{defn}

For example, if $F=K$ is a field and $A$ is an $r \times m$ matrix of rank $r$ with columns indexed by $E$, it is a classical fact 
that the function $\varphi_A$ taking an $r$-element subset of $E$ to the determinant of the corresponding $r \times r$ minor of $A$ is a Grassmann-Pl{\"u}cker function.
The function $\varphi_A$ depends (up to a non-zero scalar multiple) only on the row space of $A$, and conversely the row space of $A$ is uniquely determined by the function $\varphi_A$ (this is equivalent to the well-known fact that the {\em Pl{\"u}cker relations}
cut out the Grassmannian $G(r,m)$ as a projective algebraic set).
\medskip

We say that two Grassmann-Pl{\"u}cker functions $\varphi_1$ and $\varphi_2$ are {\bf equivalent} if $\varphi_1 = g \cdot \varphi_2$ for some $g \in F^\times$.

\begin{theorem} \label{thm:A}
Let $E$ be a non-empty finite set, let $F$ be a tract, and let $r$ be a positive integer. 
There is a natural bijection between equivalence classes of Grassmann-Pl{\"u}cker functions of rank $r$ on $E$ with coefficients in $F$ and strong $F$-matroids of rank $r$ on $E$, defined via axioms {\rm (C0)} through {\rm (C3)}.
\end{theorem}

The bijective map from equivalence classes of Grassmann-Pl{\"u}cker functions to strong $F$-matroids in Theorem~\ref{thm:A} can be described explicitly as follows.  
Let $B_\varphi$ be the {\bf support} of $\varphi$, i.e., the collection of all subsets $\{ x_1,\ldots,x_r \} \subseteq E$ such that $\varphi(x_1,\ldots,x_r) \neq 0$.  Then $B_{\varphi}$ is the set of bases for a rank $r$ matroid $M_{\varphi}$ (in the usual sense) on $E$ (cf.~\cite[Remark 2.5]{AndersonDelucchi}).
For each circuit $C$ of $M_{\varphi}$, we define a corresponding projective $F$-circuit
$X \in {\mathbb P}(F^E)$ with ${\rm supp}(X) = C$ as follows.  Let $x_0 \in C$ and let $\{ x_1,\ldots,x_r \}$ be a basis for $M_{\varphi}$ containing $C \backslash x_0$.  Then 
\begin{equation} \label{eq:CircuitsFromGP}
\frac{X(x_i)}{X(x_0)} = (-1)^i \frac{\varphi(x_0,\ldots,\hat{x}_i,\ldots,x_r)}{\varphi(x_1,\ldots,x_r)}.
\end{equation}
We will show that this is well-defined, and give an explicit description of the inverse map from strong $F$-matroids to equivalence classes of Grassmann-Pl{\"u}cker functions.

% We will also see, in Theorem~\ref{thm:3term}, that Axiom (GP3) above can be replaced by the following {\em a priori} weaker axiom:

\begin{remark}
When $F={\mathbb K}$ is the Krasner hyperfield, it is not difficult to see that (\ref{eq:GP3}) is equivalent to the following well-known condition characterizing the set of bases of a matroid (cf.~\cite[Condition (B2), p.17]{Oxley}):
\begin{itemize}
\item (Basis Exchange Axiom) Given bases $B,B'$ and $b \in B \backslash B'$, there exists $b' \in B' \backslash B$ such that $(B \cup \{ b' \}) \backslash \{ b \}$ is also a basis.
\end{itemize}
\end{remark}

\begin{defn}
 A {\bf weak Grassmann-Pl{\"u}cker function of rank $r$ on $E$ with coefficients in $F$} is a function $\varphi : E^r \to F$ such that the support of $\varphi$ is the set of bases of a rank $r$ matroid on $E$ and $\varphi$ satisfies (GP1), (GP2), and the following variant of (GP3):
\begin{itemize}
\item ${\rm (GP3)}'$ [3-term Grassmann--Pl{\"u}cker relations] Equation (\ref{eq:GP3}) holds for any two subsets $I=\{ x_1,\ldots,x_{r+1} \}$ and $J=\{ y_1,\ldots,y_{r-1} \}$ of $E$ with $|I \backslash J|=3$.
\end{itemize}
\end{defn}

It is clear that any { strong} Grassmann-Pl\"ucker function is also a weak Grassmann-Pl\"ucker function.

\begin{theorem} \label{thm:Aprime}
Let $E$ be a non-empty finite set, let $F$ be a tract, and let $r$ be a positive integer. 
There is a natural bijection between equivalence classes of weak Grassmann-Pl{\"u}cker functions of rank $r$ on $E$ with coefficients in $F$ and weak $F$-matroids of rank $r$ on $E$, defined via axioms {\rm (C0)} through {\rm (C2)} and ${\rm (C3)}'$.
\end{theorem}

% Old: \subsection{Dressians}
% \label{sec:Dressian}

\subsection{Grassmannians over hyperfields}
\label{sec:Dressian}

For concreteness and ease of notation, write $E=\{ e_1,\ldots,e_m \}$ and let $S$ denote the collection of $r$-element subsets of $\{ 1,\ldots, m\}$, so that $|S|=\binom{m}{r}$.
Given a Grassmann-Pl{\"u}cker function $\varphi$, define the corresponding {\bf Pl{\"u}cker vector} $p = (p_I)_{I \in S} \in F^S$ by $p_I := \varphi(e_{i_1},\ldots,e_{i_r})$,
where $I = \{ i_1,\ldots,i_r \}$ and $i_1 < \cdots < i_r$.  Clearly $\varphi$ can be recovered uniquely from $p$.  The vector $p$ satisfies an analogue of the Grassmann--Pl{\"u}cker relations (GP3); for example, the 3-term relations can be rewritten as follows:
for every $A \subset \{ 1,\ldots,m \}$ of size $r-2$ and $i,j,k,\ell \in \{ 1,\ldots,m \} \backslash A$, we have
\begin{equation}
\label{eq:3termGP}
p_{A \cup i \cup j} \cdot p_{A \cup k \cup \ell}  -p_{A \cup i \cup k} \cdot p_{A \cup j \cup \ell} + p_{A \cup i \cup \ell} \cdot p_{A \cup j \cup k} \in N_G.
\end{equation}

More generally, for all subsets $I,J$ of $\{ 1,\ldots, m \}$ with $|I|=r+1$, $|J|=r-1$, and $|I \backslash J| \geq 3$, the point $p=(p_I)$ lies on the ``subvariety'' of the projective space in the $\binom{m}{r}$ homogeneous variables $x_I$ for $I \in S$ defined by
\begin{equation}
\label{eq:GP}
\sum_{i \in I} {\rm sign}(i;I,J) x_{J \cup i} x_{I \backslash i} \in N_G,
\end{equation}
where ${\rm sign}(i;I,J)=(-1)^s$ with $s$ equal to the number of elements $i' \in I$ with $i < i'$ plus the number of elements $j \in J$ with $i < j$.

{ 
Although we will not explore this further in the present paper, when $F$ is a hyperfield one can view the ``equations'' (\ref{eq:GP}) as defining a hyperring scheme $G(r,m)$ 
% previously $D(r,m)$
in the sense of \cite{JunHyperringScheme}, which we call the {\bf $F$-Grassmannian}.
% \footnote{{\bf Matt:} For arbitrary tracts, one should be able to view
% $G(r,m)$ as an (ordered?) blue scheme in the sense of Lorscheid.}
% which we call the {\bf $F$-Dressian}.
In this geometric language, Theorem~\ref{thm:A} says that a strong matroid of rank $r$ on $\{ 1,\ldots, m \}$ over a hyperfield $F$ can be identified with an $F$-valued point of $G(r,m)$; thus $G(r,m)$ is a ``moduli space'' for rank $r$ matroids over $F$.
If $F=K$ is a field, the $K$-Grassmannian $G(r,m)$ coincides with the usual Grassmannian variety over $K$.
If $F={\mathbb T}$ is the tropical hyperfield, the ${\mathbb T}$-Grassmannian $G(r,m)$ is what Maclagan and Sturmfels \cite[\S{4.4}]{MaclaganSturmfels} call the {\bf Dressian} $D(r,m)$ (in order to distinguish it from a tropicalization of the Pl{\"u}cker embedding of the usual Grassmannian).
}

%%%%%%%%%%%%%%%%%%%%%%%%%%%%%%%%%%
\begin{comment}
\begin{remark} \label{rmk:Lorscheid}
Oliver Lorscheid has pointed out to us that if one works in the larger category of ordered blueprints \cite{LorscheidTropical}, which contains hyperrings as a full subcategory, admits tensor products, and has an initial object ${\mathbb F}_1$, we may identify the $F$-hyperring scheme $D(r,m)$ with the base change from ${\mathbb F}_1$ to $F$ of
a universal ordered blue scheme over ${\mathbb F}_1$, the ``${\mathbb F}_1$-Dressian''.  One could then define a matroid over an ordered blueprint $S$ to be an $S$-point of the ${\mathbb F}_1$-Dressian, generalizing our notion of matroids over hyperfields.  It seems rather unlikely, however, that there are nice generalizations
of the circuit or dual pair ``cryptomorphic'' axiom systems in this generality. 
\end{remark}
\end{comment}
%%%%%%%%%%%%%%%%%%%%%%%%%%%%%%%%%%

\subsection{Duality} \label{sec:Duality}

There is a duality theory for matroids over tracts which generalizes the established duality theory for matroids, oriented matroids, valuated matroids, etc.  (For matroids over fields, it corresponds to orthogonal complementation.)

% The dual of a matroid $M$ over a tract $F$ will be a matroid over the {\bf dual tract} $F^*$ in which $G$ is replaced by its {\em opposite group} $G^{\rm op}$.  (Formally, $G^{\rm op}$ has the same underlying set as $G$ but with group operation defined by $g_1 *_{\rm op} g_2 := g_2 * g_1$.)

\begin{theorem} \label{thm:B}
Let $E$ be a non-empty finite set with $|E|=m$, let $F$ be a tract {endowed with an involution $x \mapsto \involution{x}$},
and let $M$ be a strong (resp. weak) $F$-matroid of rank $r$ on $E$ with strong (resp. weak) $F$-circuit set ${\mathcal C}$ and Grassmann-Pl{\"u}cker function (resp. weak Grassmann-Pl{\"u}cker function) $\varphi$.
There is a strong (resp. weak) $F$-matroid $M^*$ of rank $m-r$ on $E$, called the {\bf dual matroid} of $M$, with the following properties:
\begin{itemize}
\item The $F$-circuits of $M^*$ are the elements of ${\mathcal C}^* := {\rm SuppMin}({\mathcal C}^\perp - \{ 0 \})$, where ${\rm SuppMin}(S)$ denotes the elements of $S$ of minimal support.
\item A Grassmann-Pl{\"u}cker function (resp. weak Grassmann-Pl{\"u}cker function) $\varphi^*$ for $M^*$ is defined by the formula
\[
\varphi^*(x_1,\ldots,x_{m-r}) = {\rm sign}(x_1,\ldots,x_{m-r},x_1',\ldots,x_r') \involution{\varphi(x_1',\ldots,x_r')},
\]
where $x_1',\ldots,x_r'$ is any ordering of $E \backslash \{ x_1,\ldots,x_{m-r} \}$.
\item The underlying matroid of $M^*$ is the dual of the underlying matroid of $M$, i.e., $\underline{M^*} = \underline{M}^*$.
\item $M^{**} = M$.
\end{itemize}
\end{theorem}

The $F$-circuits of $M^*$ are called the {\bf $F$-cocircuits} of $M$, and vice-versa.

%%%%%%%%%%%%%%%%%%%%%%%%%%%%%%%%%%%%%%%%%%
\begin{comment}

It follows from Theorem~\ref{thm:B} that for every hyperfield $F$ and every $F$-matroid $M$ on $E=\{ 1,\ldots,m \}$, there are vectors $v_1,\ldots,v_m \in F^t$ (for some $t$) such that the $F$-circuits of $M$ are the minimal linear relations among the $v_i$.
Indeed, if we let $w_1,\ldots,w_t \in F^m$ be the $F$-cocircuits of $M$ and we write $w_1,\ldots,w_t$ as the rows of a $t \times m$ matrix $A$, then the columns $v_1,\ldots,v_m$ of $A$ will have the desired property.
However, the set of minimal linear relations among a finite set of vectors $v_i$ with coefficients in a hyperfield $F$ need {\bf not} in general be the set of $F$-circuits of an $F$-matroid.

\begin{example}
\label{ex:matrix_indep}
Let $F = {\mathbb K}$ be the Krasner hyperfield, and let $v_1 = (1,1,0), v_2 = (1,0,1), v_3 = (1,1,1), v_4 = (0,1,0) \subset {\mathbb K}^3$.  Then $0 \in v_1 \boxplus v_2 \boxplus v_3$ and $0 \in v_2 \boxplus v_3 \boxplus v_4$, and these relations are support-minimal.  However, $\{ v_1, v_2, v_4 \}$ are linearly independent over ${\mathbb K}$
so it is not possible to eliminate $v_3$ from the two relations.  This shows that the set of $X \subset \{ 1,2,3,4 \}$ such that the vectors $v_e$ for $e \in X$ form a minimally dependent set are {\em not} the circuits of a matroid.  (Equivalently, the set of $I \subset \{ 1,2,3,4 \}$ such that the vectors $v_i$ for $i \in I$ are linearly independent
do not form the independent sets of a matroid.)
\end{example}

\end{comment}
%%%%%%%%%%%%%%%%%%%%%%%%%%%%%%%%%%%%%%%%%%

\subsection{Dual pairs}
\label{sec:DualPairs}

Let $F$ be a tract {endowed with an involution $x \mapsto \involution{x}$}, and
let $M$ be a (classical) matroid with ground set $E$.  We call a subset ${\mathcal C}$ of $F^E$ an {\bf $F$-signature of $M$} if ${\mathcal C}$ satisfies properties (C0) and (C1) from Definition~\ref{def:Fcircuits}, and taking supports gives a bijection from the projectivization of ${\mathcal C}$ to circuits of $M$.

\begin{defn}
We say that $({\mathcal C},{\mathcal D})$ is a {\bf dual pair of $F$-signatures of $M$} if:
\begin{itemize}
\item (DP1) ${\mathcal C}$ is an $F$-signature of the matroid $M$.
\item (DP2) ${\mathcal D}$ is an $F$-signature of the dual matroid $M^*$.
\item (DP3) ${\mathcal C} \perp {\mathcal D}$, meaning that $X \perp Y$ for all $X \in \cC$ and $Y \in \cD$.
\end{itemize}
\end{defn}

\begin{theorem} \label{thm:C}
Let $M$ be a matroid on $E$,  let ${\mathcal C}$ be an $F$-signature of $M$, and let ${\mathcal D}$ be an $F$-signature of $M^*$.  Then ${\mathcal C}$ and ${\mathcal D}$ are the set of $F$-circuits and $F$-cocircuits, respectively, of a strong $F$-matroid with underlying matroid $M$ if and only if 
 $({\mathcal C},{\mathcal D})$ satisfies {\rm (DP3)}  (i.e., is a dual pair of $F$-signatures of $M$).
\end{theorem}

\begin{defn}
We say that $({\mathcal C},{\mathcal D})$ is a {\bf weak dual pair of $F$-signatures of $M$} if ${\mathcal C}$ and ${\mathcal D}$ satisfy (DP1),(DP2), and the following weakening of (DP3):
\begin{itemize}
\item ${\rm (DP3)}'$ $X\perp Y$ for every pair $X \in \cC$ and $Y \in \cD$ with $|\underline{X} \cap \underline{Y}| \leq 3$.
\end{itemize}
\end{defn}

\begin{theorem} \label{thm:Cprime}
Let $M$ be a matroid on $E$,  let ${\mathcal C}$ be an $F$-signature of $M$, and let ${\mathcal D}$ be an $F$-signature of $M^*$.  Then ${\mathcal C}$ and ${\mathcal D}$ are the set of $F$-circuits and $F$-cocircuits, respectively, of a weak $F$-matroid with underlying matroid $M$ if and only if 
 $({\mathcal C},{\mathcal D})$ satisfies ${\rm (DP3)}'$  (i.e., is a weak dual pair of $F$-signatures of $M$).
\end{theorem}

\begin{comment}
\begin{remark}
Suppose that we are given an involution $\tau \colon x \mapsto \overline x$ of our hyperfield $F$. Then we could get an equivalent duality theory in which we apply the involution to all elements of the hyperfield assigned to the dual matroid. Thus the dual of a Grassmann-Pl\"ucker function would be given by the formula 
\[
\varphi^*(x_1,\ldots,x_{m-r}) = {\rm sign}(x_1,\ldots,x_{m-r},x_1',\ldots,x_r') \overline{\varphi(x_1',\ldots,x_r')},
\]
and the inner product giving the orthogonality between circuits and cocircuits would be $X \odot Y := (x_1 \odot \overline{y}_1) \boxplus \cdots \boxplus (x_m \odot \overline{y}_m)$. The advantage of this approach is that this scalar product is more natural in certain cases, such as when $F$ is the field of complex numbers or the phase hyperfield and the involution is given by conjugation. This is the approach taken by Anderson and Delucchi in \cite{AndersonDelucchi}. Since, however, the theory with the involution is completely equivalent to that without, for the sake of notational simplicity we shall not make use of such an involution in this paper.
\end{remark}
\end{comment}

\subsection{Minors}
\label{sec:minors}

Let ${\mathcal C}$ be the set of $F$-circuits of a (strong or weak) $F$-matroid $M$ on $E$, and let $A \subseteq E$.
For $X \in {\mathcal C}$, define $X \backslash A \in F^{E \backslash A}$ by $(X \backslash A)(e) = X(e)$ for $e \not\in A$.
(Thus $X \backslash A$ can be thought of as the restriction of $X$ to the complement of $A$.)

Let ${\mathcal C} \backslash A = \{ X \backslash A \; | \; X \in {\mathcal C}, \; \underline{X} \cap A = \emptyset \}$. 
Similarly, let ${\mathcal C} / A = {\rm SuppMin}(\{ X \backslash A \; | \; X \in {\mathcal C} \})$.  

\begin{theorem} \label{thm:D}
Let ${\mathcal C}$ be the set of $F$-circuits of a strong (resp. weak) $F$-matroid $M$ on $E$, and let $A \subseteq E$.
Then ${\mathcal C} \backslash A$ is the set of $F$-circuits of a strong (resp. weak) $F$-matroid $M \backslash A$ on $E \backslash A$, called the {\bf deletion} of $M$ with respect to $A$, whose underlying matroid is $\underline{M} \backslash A$.
Similarly, ${\mathcal C} / A$ is the set of $F$-circuits of a strong (resp. weak) $F$-matroid $M / A$ on $E \backslash A$, called the {\bf contraction} of $M$ with respect to $A$, whose underlying matroid is $\underline{M} / A$.
Moreover, we have $(M \backslash A)^* = M^* / A$ and $(M / A)^* = M^* \backslash A$.
\end{theorem}

\subsection{Equivalence of different definitions}
\label{sec:whythesame}

We briefly indicate how to see the equivalence of various flavors of matroids in the literature with our notions of strong and weak $F$-matroid, for some specific choices of the tract $F$.

\begin{example} \label{ex:fieldexample}
When $F=K$ is a field, a strong or weak matroid on $E$ with coefficients in $K$ is the same thing as a vector subspace of $K^E$ in the usual sense.  
% Indeed, the equivalence of (V0)-(V3) with (GP1)-(GP3) is a classical fact about the Grassmannian $G(r,m)$ (see e.g. \cite{KleimanLaksov}), and the equivalence of (GP1)-(GP3) with (C0)-(C3) is part of Theorem~\ref{thm:maintheorem}.
Indeed, a weak Grassmann-Pl{\"u}cker function with coefficients in a field $K$ automatically satisfies (GP3) (cf.~the proof of \cite[Theorem 1]{KleimanLaksov}), and the bijection between $r$-dimensional subspaces of $K^E$ and equivalence classes of rank $r$ Grassmann-Pl{\"u}cker functions with coefficients in $K$ also follows from 
{\em loc.~cit.}
\end{example}

\begin{example}
A strong or weak matroid over ${\mathbb K}$ is 
% essentially 
the same thing as a matroid in the usual sense.  
% This follows, for example, from \cite[Lemma A.3]{AndersonDelucchi}.
\end{example}

\begin{example}
A strong or weak matroid over ${\mathbb T}$ is the same thing as a valuated matroid in the sense of Dress--Wenzel \cite{DressWenzelVM}.
This follows from \cite[Theorem 3.2]{MurotaTamura} and the discussion at the top of page 202 in {\em loc.~cit.}
\end{example}

\begin{example} \label{ex:signsexample}
A strong or weak matroid over ${\mathbb S}$ is the same thing as an {\bf oriented matroid} in the sense of Bland--Las Vergnas \cite{BlandLasVergnas}.
% This follows from \cite[Theorem 3.6.1]{OM}; see also Theorem 3.4.3 and Corollary 3.5.12 of {\em loc.~cit.}
This follows for example from \cite[Theorems 3.5.5 and 3.6.2]{OM}.
\end{example}

\begin{example} \label{ex:weaksignsexample}
A strong or weak matroid over ${\mathbb W}$ is the same thing as a weakly oriented matroid in the sense of Bland and Jensen \cite{BlandJensen}.  This follows from the results of {\em loc.~cit.}
\end{example}

% \begin{example}
% A matroid over ${\mathbb P}$ is the same thing as an {\bf phased matroid}\footnote{These are called {\it complex matroids} in \cite{AndersonDelucchi} but the authors of that paper now prefer the term ``phased matroid''.} in the sense of Anderson--Delucchi \cite{AndersonDelucchi} (see for example Theorem A of {\em loc.~cit.})
% \end{example}

\begin{example} \label{regularexample}
A strong or weak matroid $M$ over a partial field $P$ is the same thing as a representation of $\underline{M}$ over $P$ in the sense of
\cite[Definition 2.4]{PvZSkew}.  This follows from \cite[Proof of Theorem 1]{KleimanLaksov} exactly as in Example~\ref{ex:fieldexample}, since that
argument works verbatim if one replaces the field of coefficients by a partial field.
In particular, a (strong or weak) matroid over the regular partial field ${\mathbb U}_0$ is the same thing as a regular matroid.  
\end{example}

\begin{example}\label{ex:initial}
A weak matroid over ${\mathbb I}$ is the same thing as a weak matroid over ${\mathbb U}_0$, since these tracts have the same underlying multiplicative group and their null sets contain the same elements having at most 3 summands. On the other hand, the matroids strongly representable over ${\mathbb I}$ are precisely the direct sums of matroids of the form $M(G)$ with $G$ a series-parallel network. Indeed, it is clear that no matroid strongly representable over ${\mathbb I}$ can have a circuit-cocircuit intersection of size 4. Hence such matroids also cannot have minors with such circuit-cocircuit intersections. In particular, they can have no $M(K_4)$-minors. But it is known that the only regular connected matroids with no $M(K_4)$-minor are those of the form $M(G)$ with $G$ a series-parallel network \cite[Corollary 11.2.15]{Oxley}. It is easy to check that any such matroid is strongly representable over ${\mathbb I}$.
\end{example}

\subsection{Weak $F$-matroids which are not strong $F$-matroids}
\label{sec:counterexample}
Even for hyperfields $F$, there are many examples of a weak $F$-matroids which are not strong $F$-matroids. Our first example of this phenomenon is over the triangle hyperfield.
% Here is an example of a weak $F$-matroid which is not a strong $F$-matroid.

\begin{example}
Let $F$ be the triangle hyperfield ${\mathbb V}$ (cf. Example~\ref{ex:triangle}).
Consider the Grassmann-Pl\"ucker function $\varphi$ of rank 3 on the 6 element set $E = \{1,2,\ldots 6\}$ with coeffiecients in $F$ given by
$$\varphi(x_1, x_2, x_3) = \begin{cases}4 &\text{if }\{x_1,x_2,x_3\} = \{1,5,6\}\\2 &\text{if $\{x_1, x_2,x_3\}$ consists of one element from each of}\\ &\text{$\{1\}$, $\{2, 3, 4\}$ and $\{5,6\}$}\\ 1 &\text{otherwise, as long as the $x_i$ are distinct}\\0 &\text{if the $x_i$ are not distinct.}\end{cases}$$ Then $\varphi$ is symmetric under permutation of the second, third and fourth coordinates, as well as under exchange of the fifth and sixth coordinates. It is clear that $\varphi$ satisfies (GP1) and (GP2), and the support of $\varphi$ is the set of bases of the uniform matroid $U_{3,6}$ of rank 3 on 6 elements. 
We will now verify that $\varphi$ also satisfies ${\rm (GP3)}'$. 

\medskip

Suppose we have subsets $I = \{x_1, x_2, x_3, x_4\}$ and $J = \{y_1,y_2\}$ of $E$ with $|I \setminus J| = 3$. Then $|I \cap J| = 1$. So all summands in the corresponding 3-term Grassmann-Pl\"ucker relation are in the set $\{0,1,2,4\}$. In order to show that the relation holds, it suffices to show that there cannot be summands equal to each of 1 and 4. So suppose for a contradiction that $\varphi(x_2,x_3,x_4) \odot \varphi(x_1, y_1, y_2) = 1$ but $\varphi(x_1, x_3, x_4) \odot \varphi(x_2, y_1, y_2) = 4$. Then we have $\varphi(x_2,x_3,x_4) = \varphi(x_1,y_1,y_2) = 1$. Since $\varphi(x_1, y_1, y_2)$ and $\varphi(x_2, y_1, y_2)$ are nonzero, neither $x_1$ nor $x_2$ is in $J$. If $\varphi(x_1, x_3, x_4) = 4$ then $x_1 = 1$ and $\{x_3, x_4\} = \{5,6\}$, but neither 5 nor 6 can be in $J$ since $\varphi(x_1, y_1, y_2) = 1$. So $I$ and $J$ are disjoint, which is impossible. If $\varphi(x_2, y_1, y_2) = 4$ then $x_2 = 1$ and $\{y_1, y_2\} = \{5,6\}$, but neither 5 nor 6 can be in $\{x_3, x_4\}$ since $\varphi(x_2, x_3, x_4) = 1$. So once more $I$ and $J$ are disjoint, which is impossible. The only remaining case is that $\varphi(x_1,x_3,x_4) = \varphi(x_2, y_1, y_2) = 2$. Thus both of these sets contain 1, so that without loss of generality $y_1 = 1$. Then $y_2 \not \in \{5,6\}$, so $x_2 \in \{5,6\}$. Thus neither of $x_3$ or $x_4$ can be 1, so $x_1 = 1$, so that $\varphi(x_1, y_1, y_2) = 0$, again a contradiction.

\medskip

On the other hand, not all of the Grassmann-Pl\"ucker relations are satisfied. Let $I = \{1,2,3,4\}$ and $J= \{5,6\}$. The corresponding Grassmann-Pl\"ucker relation is $0 \in (1 \odot 4) \boxplus (1 \odot 1) \boxplus (1 \odot 1) \boxplus (1 \odot 1)$, which is false.
\end{example}

Our second example is due to Daniel Wei\ss auer, and it shows that weak and strong matroids do not coincide over the phase hyperfield. It has only been verified by an exhaustive computer check, and so we do not provide a proof here. 

\begin{example}\label{eg:danscex}
Consider the weak Grassmann-Pl\"ucker function of rank 3 on the 6-element set $\{x,y,z,t,l,m\}$ given by
$$\begin{array}{rclrclrcl}
\varphi(x,y,z) &=& 1&
\varphi(x,y,t) &=& -1&
\varphi(x,z,t) &=& 1\\
\varphi(y,z,t) &=& -1&
\varphi(x,y,l) &=& e^{(0.9+\pi)i}&
\varphi(x,z,l) &=& e^{2.5i}\\
\varphi(y,z,l) &=& e^{5.5i}&
\varphi(x,t,l) &=& e^{(2.7+\pi)i}&
\varphi(y,t,l) &=& e^{(5.8-\pi)i}\\
\varphi(z,t,l) &=& e^{(0.3 + \pi)i}&
\varphi(x,y,m) &=& e^{(0.5 + \pi)i}&
\varphi(x,z,m) &=& e^{1.2i}\\
\varphi(y,z,m) &=& e^{3.8i}&
\varphi(x,t,m) &=& e^{(3 + \pi)i}&
\varphi(y,t,m) &=& e^{(5.1-\pi)i}\\
\varphi(z,t,m) &=& e^{(0.4 + \pi)i}&
\varphi(x,l,m) &=& e^{3.1i}&
\varphi(y,l,m) &=& e^{0.1i}\\
\varphi(z,l,m) &=& 1&
\varphi(t,l,m) &=& e^{3.1i}&&&
\end{array}$$
and with the remaining values determined by (GP2). This function satisfies ${\rm (GP3)}'$, but it does not satisfy (GP3). Consider for example the lists $(x,y,z,t)$ and $(l,m)$. Applying (GP3) to these lists gives $e^{3.1i} \oplus e^{0.1i} \oplus 1 \oplus e^{3.1i} \ni 0$, which is false.
\end{example}

%%%%%%%%%%%%%%%%%%%%%%%%%%%%%%%%%%%%%%%%%%%%%%%%%%%%%%%%%%%%%%%%%%%%%%%%%%%%%%%%%%%%%%%%%%%%%%%%%%%%%%%

%%%%%%%%%

\subsection{Functoriality}% and representability
\label{sec:functoriality}

In this section we discuss the behavior of matroids over tracts with respect to homomorphisms of the latter.
% concept of {\em realizability} of matroids over tracts in the general context of push-forward maps.

% \subsection{Push-forwards and realizability}

Recall that if $F$ is a tract and $M$ is an $F$-matroid on $E$, there is an underlying classical matroid $\underline{M}$, and that classical matroids are the same as matroids over the Krasner hyperfield ${\mathbb K}$.  We now show that the ``underlying matroid'' construction is a special case of a general push-forward operation on matroids over tracts.

The following lemma is straightforward from the various definitions involved:

\begin{lemma} \label{lem:pushforward}
If $f : F \to F'$ is a homomorphism of tracts and $M$ is a strong (resp. weak) $F$-matroid on $E$, 
\[
\{ c' f_*(X) \; : \; c' \in (F')^\times, \; X \in \cC(M) \}
\]
% PREVIOUS WRONG DEFINITION: the image under the induced map $f_* : F^E \to (F')^E$ of the set of circuits of $M$ 
is the set of $F'$-circuits of a strong (resp. weak) $F'$-matroid $f_*(M)$ on $E$, called the {\bf push-forward} of $M$.
\end{lemma}

% \begin{remark} \label{rmk:pushforward}
The tract associated to the Krasner hyperfield ${\mathbb K}$ (which by abuse of terminology we also denote by ${\mathbb K}$) is a final object
in the category of tracts.  
Indeed, recall that the tract $(H,N_H)$ associated to ${\mathbb K}$ has $H = \{ 1 \}$, $\NN[H] = \NN$, and $N_H = \NN \backslash \{ 1 \}$. 
Thus if $F=(G,N_G)$ is a tract, there is a unique homomorphism $\psi : F \to {\mathbb K}$ sending $0$ to $0$ and every element of $G=F^\times$ to $1$.
$0$ to $0$ and all non-zero elements of $F$ to $1$.  

\medskip

If $M$ is an $F$-matroid, the push-forward $\psi_*(M)$ coincides with the underlying matroid $\underline{M}$.

\medskip

Given a Grassmann-Pl{\"u}cker function (resp. weak Grassmann-Pl{\"u}cker function) $\varphi : E^r \to F$ and a homomorphism of tracts $f : F \to F'$, we define the {\bf push-forward} $f_* \varphi : E^r \to F'$ by the formula
\[
(f_* \varphi)(e_1,\ldots,e_r) = f(\varphi(e_1,\ldots,e_r)).
\]
This is easily checked to once again be a Grassmann-Pl{\"u}cker function (resp. weak Grassmann-Pl{\"u}cker function).

\medskip

As an immediate consequence of (\ref{eq:CircuitsFromGP}), we see that the push-forward of an $F$-matroid can be defined using either circuits or Grassmann-Pl{\"u}cker functions:

\begin{lemma} 
\label{lem:compatibility_of_pushforwards}
If $M_{\varphi}$ is the strong (resp. weak) $F$-matroid associated to the Grassmann-Pl{\"u}cker function (resp. weak Grassmann-Pl{\"u}cker function) 
$\varphi : E^r \to F$, and $f : F \to F'$ is a homomorphism of tracts, then $f_*(M_{\varphi}) = M_{f_* \varphi}$.
\end{lemma}

It is also straightforward to check (using either circuits or Grassmann-Pl{\"u}cker functions) that if $M$ is a strong (resp. weak) $F$-matroid and $f : F \to F'$ is a homomorphism of tracts,
% with involution (meaning we require that $\overline{f(x)} = f(\overline{x})$ for all $x \in F$)
then the dual strong (resp. weak) $F'$-matroid to $f_*(M)$ is $f_*(M^*)$.  Summarizing our observations in this section, we have:

\begin{corollary}
If $M$ is a strong (resp. weak) $F$-matroid with $F$-circuit set $\cC(M)$ and Grassmann-Pl{\"u}cker function (resp. weak Grassmann-Pl{\"u}cker function) $\varphi$, and $f : F \to F'$ is a homomorphism of tracts% with involution
, the following coincide:
\begin{enumerate}
\item The strong (resp. weak) $F'$-matroid whose $F'$-circuits are $\{ c' f_*(X) \; : \; c' \in (F')^\times, \; X \in \cC(M) \}$.
\item The strong (resp. weak) $F'$-matroid whose $F'$-cocircuits are $\{ c' f_*(Y) \; : \; c' \in (F')^\times, \; Y \in \cC(M^*) \}$.
\item The strong (resp. weak) $F'$-matroid whose Grassmann-Pl{\"u}cker function (resp. weak Grassmann-Pl{\"u}cker function) is $f_* \varphi$.
\end{enumerate}
\end{corollary}

\begin{defn} \label{defn:realizable}
Let $f : F \to F'$ be a homomorphism of tracts, and let $M'$ be a strong (resp.~weak) matroid on $E$ with coefficients in $F'$.  We say that $M'$ is {\bf realizable with respect to $f$} if there is a strong (resp.~weak) matroid $M$ over $F$ such that $f_*(M)=M'.$

If $F'={\mathbb K}$ is the Krasner hyperfield, so that $M'$ is a matroid in the usual sense, we say that $M'$ is {\bf strongly realizable over $F$} (resp.~{\bf weakly realizable over $F$}) if there is a strong (resp.~weak) matroid $M$ over $F$ such that $\psi_*(M)=M'$, where $\psi : F \to {\mathbb K}$ is the canonical homomorphism.
\end{defn}

%%%%%%%%%%%%%%%%%%%%%%%%%%%%%%%%%%%%%%%%%%%%%%%%%%%%%%%%%%%%%

\subsection{Perfect tracts and doubly distributive partial hyperfields}
\label{sec:perfect}

Although the notions of weak and strong matroids over tracts do not coincide in general, they do agree for a special class which we call {\em perfect tracts}.
As a key example, the tracts associated to {\em doubly distributive partial hyperfields} are perfect; this follows from some results of Dress and Wenzel in \cite{DressWenzelPM}, as we will explain in this section.
% \footnote{{\bf Matt:} We may want to replace the reference to the work of Dress--Wenzel with a self-contained proof based on the structure theorem, and to generalize part or all of this discussion from hyperfields to partial hyperfields.}

We say that a partial hyperfield $P$ is {\bf doubly distributive} if for any $x$, $y$, $z$ and $t$ in $P$ we have $(x \boxplus y) (z \boxplus t) = xz \boxplus xt \boxplus yz \boxplus yt$. It follows that $$\left(\bigboxplus_{i \in I} x_i  \right)\left(\bigboxplus_{j \in J} y_j \right) = \bigboxplus_{\substack{i \in I \\ j \in J}} x_i y_j$$ for any finite families $(x_i)_{i \in I}$ and $(y_j)_{j \in J}$. Not all (partial) hyperfields have this property. For example, the triangle and phase hyperfields are not doubly distributive, whereas the Krasner, sign and tropical hyperfields are.

% Dress and Wenzel do not work directly with hyperfields, but rather with objects called fuzzy rings.
We can build a fuzzy ring in the sense of Dress--Wenzel from a doubly distributive partial hyperfield $P=(G,R)$ by setting $K = \NN[G]$, $\epsilon = -1$, and $K_0 = \{ \sum g_i \in K \; : \; 0 \in \boxplus g_i \}$.  (Double distributivity is needed to verify axiom (FR7) from \S\ref{sec:fuzzyringtract}.)

This fuzzy ring is in fact a fuzzy integral domain, in the sense that for $x,y \in K$ with $x \cdot y \in K_0$ we have either $x \in K_0$ or 
$y \in K_0$, and is distributive in the sense that $x \cdot (y_1 + y_2) = x \cdot y_1 + x \cdot y_2$.

Furthermore, if we have a strong matroid $M$ over $F$ on a set $E$ with $F$-circuit set $\mathcal C$, then $\mathcal C$ presents a matroid with coefficients in $K$ in the sense of \cite{DressWenzelPM}. This is not completely obvious: in order to prove it, we must analyze a key operation from that paper. Let $r, s \in K^E$ and let $f \in E$. Then we define $r \wedge_f s \in K^E$ by $$(r \wedge_f s)(e) := \begin{cases}0 & \text{if } e = f \\ s(f) \cdot r(e) + (-1) \cdot r(f) \cdot s(e) & \text{if } e \neq f.\end{cases}$$
For $r, s \in K^E$ we write $r \bot s$ to mean $\sum_{e \in E} r(e) \cdot s(e) \in K_0$. Let $\mathcal C^*$ be the set of $F$-cocircuits of $M$. We say that $r \in K^E$ is a {\bf fuzzy vector} of $M$ if $r \bot Y$ for any $Y \in \mathcal C^*$. It follows from Lemma 2.4(i) of \cite{DressWenzelPM} that if $r$ and $s$ are fuzzy vectors then so is $r \wedge_f s$ for any $f \in E$. Since all elements of $\mathcal C$ are fuzzy vectors, the following lemma suffices to establish that $\mathcal C$ presents a matroid with coefficients in $K$:\footnote{We omit the definition, since it is a bit technical.}

\begin{lemma} \label{lem:fuzzypresents}
Let $r$ be a fuzzy vector of $M$ and choose $e \in E$ with $r(e) \not \in K_0$. Then there is some $X \in \mathcal C$ with $e \in \underline X \subseteq \underline r$.
\end{lemma}
\begin{proof}
For any $Y \in \mathcal C^*$ we have $\underline r \cap \underline Y \neq \{e\}$ since $r \bot Y$, so there is no cocircuit of $\underline M$ which meets $\underline r$ only in $\{e\}$. Thus $e$ is not a coloop of $\underline M | \underline r$, and so there is some circuit $C$ of $\underline M | \underline r$ containing $e$. It suffices to take $X$ to be any element of $\mathcal C$ with $\underline X = C$.
\end{proof}

Furthermore, $\mathcal C^*$ presents the dual matroid with coefficients to the one presented by $\mathcal C$.

Now we can apply Theorem 2.7 of \cite{DressWenzelPM} to obtain:

\begin{theorem}
Let $M$ be a strong matroid over a doubly distributive partial hyperfield $P$, let $r$ be a fuzzy vector of $M$ and let $s$ be a fuzzy vector of $M^*$. Then $r \bot s$.
\end{theorem}

This has an important consequence which can be expressed without reference to the fuzzy ring $K$. Given a strong matroid $M$ over a tract $F$ on a set $E$ with $F$-circuit set $\mathcal C$ and $F$-cocircuit set $\mathcal C^*$, a {\bf vector} of $M$ is an element of $F^E$ which is orthogonal to everything in $\Ccal^*$. Similarly a {\bf covector} of $M$ is an element of $F^E$ which is orthogonal to everything in $\Ccal$. We say that $F$ is {\bf perfect} if, for any strong matroid $M$ over $F$, all vectors are orthogonal to all covectors.

\begin{cor}
Any doubly distributive partial hyperfield is perfect.
\end{cor}

Note, however, that there are perfect hyperfields which are not doubly distributive. For example, it is not hard to check that any weak hyperfield $W(G,\epsilon)$ is perfect, since any hypersum with more than 3 nonzero summands in a weak hyperfield contains 0.

% We are now in a position to show that the notions of strong and weak matroids coincide over doubly distributive partial hyperfields and over weak hyperfields, or more generally over perfect tracts.

We will show the following in \S\ref{sec:perfectproof}:\footnote{For fuzzy rings, this also follows from Theorem 3.4 of \cite{DressWenzelPM}, but our argument is different.}

\begin{theorem}\label{thm:perfect}
Any weak matroid $M$ over a perfect tract $F$ is strong.
\end{theorem}

%We begin by briefly discussing minors of weak $F$-matroids. Let $M$ be a weak $F$-matroid with $F$-circuit set $\Ccal$, and let $P$ and $Q$ be disjoint subsets of the ground set $E$ of $M$. Then we can define a new weak $F$-matroid $M/P\backslash Q$ on $E \setminus (P \dot \cup Q)$ by taking the set of $F$-circuits to be the set $\Ccal/P\backslash Q$ of elements of $$\{X \restric_{E \setminus (P \dot \cup Q)} | X \in \Ccal \text{ and } \underline X \cap Q = \emptyset\}$$ with minimal nonempty support. This has all the nice properties we might hope for: $(M/P_1 \backslash Q_1)/P_2 \backslash Q_2 = M/(P_1 \dot \cup P_2) \backslash (Q_1 \dot \cup Q_2)$ and $(M/P \backslash Q)^* = M^* /Q \backslash P$ and $\underline{M/P \backslash Q} = \underline M /P \backslash Q$. 

\section{Proofs}
\label{sec:proofsection}

In this section, we provide proofs of the main theorems of the paper.
We closely follow the arguments of Anderson--Delucchi from
\cite{AndersonDelucchi}; when the proof is a straightforward modification of a corresponding result in {\em loc.~cit.}, we sometimes omit details.

In order to simplify the notation, we assume throughout this section that the involution $\tau: x \mapsto \overline{x}$ is trivial.\footnote{{ This is in fact a harmless assumption, since one can deduce the general case of the theorems in \S\ref{sec:Duality} and \ref{sec:minors} from this special one.
To see this, first suppose we have proved Theorem~\ref{thm:B} in the special case $\tau = {\rm id}$.  Then Theorem~\ref{thm:B} for $(M,\tau)$ follows from the special case $(\overline{M},{\rm id})$, where $\overline{M}$ is the matroid whose $F$-circuits are obtained by replacing each $F$-circuit $C$ of $M$ with its image $\overline{C}$ under $\tau$.  The other theorems in \S\ref{sec:Duality} and \ref{sec:minors} follow similarly.}}

\subsection{Weak Grassmann-Pl{\"u}cker functions and Duality}

Given a weak Grassmann-Pl{\"u}cker function $\varphi : E^r \to F$ of rank $r$ on the ground set $E$, we set 
\[
{\mathbf B}_\varphi := \{ \{ b_1,\ldots,b_r \} \; | \; \varphi(b_1,\ldots,b_r) \neq 0 \}.
\]
Recall that this is the set of bases of a matroid of rank $r$, which we denote by $\underline{M}_\varphi$ (rather than the typographically more awkward $\underline{M_\varphi}$) and call the {\bf underlying matroid} of $\varphi$.

In what follows, we fix a total order on $E$. Let $|E| = m$.

\begin{defn}
Let $\varphi$ be a rank $r$ weak Grassmann-Pl{\"u}cker function on $E$, and for every ordered tuple $(x_1,x_2,\ldots,x_{m-r}) \in E^{m-r}$ let 
$x_1',\ldots,x_r'$ be an ordering of $E \backslash \{ x_1,x_2,\ldots,x_{m-r} \}$.  Define the {\bf dual weak Grassmann-Pl{\"u}cker function} $\varphi^* : E^{m-r} \to F$ by 
\[
\varphi^*(x_1,\ldots,x_{m-r}) := {\rm sign}(x_1,\ldots,x_{m-r},x_1',\ldots,x_r') \noinvolution{\varphi(x_1',\ldots,x_r')}. 
\]
\end{defn}

Note that, up to a global change in sign, $\varphi^*$ is independent of the choice of ordering of $E$.

%%%%%%%%%%%%%%
\begin{comment}
\begin{remark}
Our definition differs slightly from \cite[Definition 3.1]{AndersonDelucchi} in that they use $\varphi(x_1',\ldots,x_r')^{-1}$ in the above definition instead of 
$\varphi(x_1',\ldots,x_r')$.  This is because in Definition 2.12, Anderson and Delucchi use an inner product modeled on the standard Hermitian inner product on $\CC^m$ to define the notion of orthogonality, whereas we use an inner product modeled on the standard inner product on $\RR^m$.  
The proofs in \cite{AndersonDelucchi} actually become a bit simpler when one uses our modified definition.
\end{remark}
\end{comment}
%%%%%%%%%%%%%%

\begin{lemma} \label{lem:Lemma3.2}
$\varphi^*$ is a rank $(m-r)$ weak Grassmann-Pl{\"u}cker function over $F$, and the underlying matroid $\underline{M}_{\varphi^*}$ is the matroid dual of $\underline{M}_\varphi$. If $\varphi$ is a Grassmann-Pl\"ucker function then so is $\varphi^*$.
\end{lemma}

\begin{proof}
The fact that ${\mathbf B}_{\varphi^*}$ is the set of bases for  $\underline{M}_\varphi^*$ is immediate from the definitions.  To see that $\varphi^*$ is a rank $(m-r)$ weak Grassmann-Pl{\"u}cker function (resp. Grassmann-Pl{\"u}cker function), it suffices to prove (GP3)$'$ (resp. (GP3)) since (GP1) and (GP2) are clear.  This also follows from \cite[Proof of Lemma 3.2]{AndersonDelucchi}.
% Suppose $X := \{ x_0,\ldots,x_{m-r} \}$ and $Y :=  \{ y_1,\ldots,y_{m-r-1} \}$, numbered so that 
% $X \cap Y = \{ x_{m-r-\ell},\ldots,x_{m-r} \}  = \{ y_1,\ldots,y_\ell \}$.  Choose a total order $A$ on $E \backslash (X \cap Y)$, and let 
% \[
% x_0,\ldots,x_{m-r},y_{\ell + 1},\ldots,y_{m-r-1},A
% \]
% be the corresponding ordering of $E$.  Then, by the proof of \cite[Lemma 3.2]{AndersonDelucchi}, we have
% \[
% \begin{aligned}
% \varphi^*(x_0,\ldots,\hat{x}_k,\ldots,x_{m-r}) \odot \varphi^*(x_k,y_1,\ldots,y_{m-r-1}) \\ = \sigma \odot \varphi(x_k,y_{\ell + 1},\ldots,y_{m-r-1},A) \odot \varphi(x_0,\ldots,\hat{x}_k,\ldots,x_{m-r-1},A) \\
% \end{aligned}
% \]
% where 
% \[
% \sigma = {\rm sign}(x_0,\ldots,x_{m-r},y_{\ell + 1},\ldots,y_{m-r-1},A) \odot {\rm sign}(y_1,\ldots,y_{m-r-1},x_0,\ldots,x_{m-r-\ell},A).
% \]
% This implies the desired result.
\end{proof}

\subsection{Weak Grassmann-Pl{\"u}cker functions, Contraction, and Deletion}

Let $\varphi$ be a rank $r$ weak Grassmann-Pl{\"u}cker function on $E$, and let $A \subset E$.  

\begin{defn}
\begin{enumerate}
\item (Contraction)
Let $\ell$ be the rank of $A$ in $\underline M_{\varphi}$, and let $\{a_1,a_2,\ldots,a_\ell \}$ be a maximal $\varphi$-independent subset of $A$.  Define $\varphi / A : (E \backslash A)^{r - \ell} \to F$ by
\[
(\varphi / A)(x_1,\ldots,x_{r - \ell}) := \varphi(x_1,\ldots,x_{r - \ell},a_1,\ldots,a_\ell).
\]
\item (Deletion)
Let $k$ be the rank of $E \backslash A$ in $\underline M_{\varphi}$, and choose $a_1,\ldots,a_{r-k} \subseteq A$ such that $\{ a_1,\ldots,a_{r-k} \}$ is a basis of $\underline M_{\varphi}/(E\setminus A)$.  Define $\varphi \backslash A : (E \backslash A)^{k} \to F$ by
\[
(\varphi \backslash A)(x_1,\ldots,x_k) := \varphi(x_1,\ldots,x_k,a_1,\ldots,a_{r-k}).
\]
\end{enumerate}
\end{defn}

Note that for different choices of the $a_i$ the objects defined here may be scaled by a constant factor, so that strictly speaking these operations are defined only for scaling-equivalence classes of Grassmann-Pl\"ucker functions. 

The proof of the following lemma is the same as the proofs of Lemmas 3.3 and 3.4 of  \cite{AndersonDelucchi}:

\begin{lemma} \label{lem:MinorLemma}
\begin{enumerate}
\item Both $\varphi / A$ and $\varphi \backslash A$ are weak Grassmann-Pl{\"u}cker functions, and they are Grassmann-Pl\"ucker functions if $\varphi$ is. Their definitions are independent of all choices up to global multiplication by a nonzero element of $F$.  
\item $\underline{M}_{\varphi / A} = \underline{M}_{\varphi} / A$ and $\underline{M}_{\varphi \backslash A} = \underline{M}_{\varphi} \backslash A$.
\item $(\varphi \backslash A)^* = \varphi^* / A$.
\end{enumerate}
\end{lemma}

\subsection{Dual Pairs from Grassmann-Pl{\"u}cker functions}

Let $\varphi$ be a rank $r$ weak Grassmann-Pl{\"u}cker function on $E$ with underlying matroid $\underline M_{\varphi}$.

\begin{lemma}\label{lem:blergh}
Let $C$ be a circuit of $M_{\varphi}$, and let $e,f \in C$.  The quantity
\[
\frac{\varphi(e,x_2,\ldots,x_r)}{\varphi(f,x_2,\ldots,x_r)}
%  := \varphi(e,x_2,\ldots,x_r) \odot \varphi(f,x_2,\ldots,x_r)^{-1}
\]
is independent of the choice of $x_2,\ldots,x_r$ such that $\{ f,x_2,\ldots,x_r \}$ is a basis for $M_\varphi$ containing $C \backslash e$.  
\end{lemma}

\begin{proof}
(cf.~\cite[Lemma 4.1]{AndersonDelucchi}) 
Let $\{ f,x_2,\ldots,x_{r-1},x_r' \}$ be another basis for $M_\varphi$ containing $C \backslash e$.  
By Axiom (GP3)$'$, we have
\[
\varphi(f,x_2,\ldots,x_r) \cdot \varphi(e,x_2,\ldots,x_{r-1},x_r') -\varphi(e,x_2,\ldots,x_r) \cdot \varphi(f,x_2,\ldots,x_{r-1},x_r') \in N_G
\]
which implies, by Lemma~\ref{lem:negatives}, that
% by Axiom (H1) in Definition~\ref{def:hypergroup}, that 
\[
\varphi(f,x_2,\ldots,x_r) \cdot \varphi(e,x_2,\ldots,x_{r-1},x_r') = \varphi(e,x_2,\ldots,x_r) \cdot \varphi(f,x_2,\ldots,x_{r-1},x_r').
\]
This proves the lemma for $\varphi$-bases which differ by a single element, and the general case follows by induction on the number of elements by which two chosen bases differ.
\end{proof}

\begin{defn} 
\label{defn:Cphi}
Define ${\mathcal C}_\varphi$ to be the collection of all $X \in F^E$ such that: 
\begin{enumerate}
\item $\underline{X}$ is a circuit of $\underline{M_{\varphi}}$
\item For every $e,f \in E$ and every basis $B=\{ f,x_2,\ldots,x_r \}$ with $\underline X \backslash e \subseteq B$, we have
\[
\frac{X(f)}{X(e)} = -\frac{\varphi(e,x_2,\ldots,x_r)}{\varphi(f,x_2,\ldots,x_r)}.
\]
% (Recall from (\ref{eq:convention}) that $\frac{a}{b}$ means $b^{-1}a$ when $F^\times$ is non-commutative.)
\end{enumerate}
\end{defn}

It is easy to see that ${\mathcal C}_\varphi$ depends only on the equivalence class of $\varphi$.  Set ${\mathcal D}_\varphi := {\mathcal C}_{\varphi^*}$.

\begin{lemma} \label{lem:Prop4.3}
\begin{enumerate}
\item The sets ${\mathcal C}_\varphi$ and ${\mathcal D}_\varphi$ form a weak dual pair of $F$-signatures of $M_\varphi$ in the sense of \S\ref{sec:DualPairs}.
\item If $\varphi$ is a Grassmann-Pl\"ucker function then ${\mathcal C}_\varphi$ and ${\mathcal D}_\varphi$ form a dual pair.
\item ${\mathcal C}_{\varphi / e} = \cC_\varphi /e$ and ${\mathcal C}_{\varphi \backslash e} = \cC_\varphi \backslash e$.
\end{enumerate}
\end{lemma}

\begin{proof}
(cf.~\cite[Proposition 4.3]{AndersonDelucchi}) 
We begin by showing that every circuit $C$ of $\underline{M_{\phi}}$ is the support of an element of $\cC_{\varphi}$. Let $y_0$ be any element of $C$ and let $\{y_1 \ldots y_r\}$ be any basis of $\underline{M_{\phi}}$ extending $C\setminus y_0$. Define $X(y_i) = (-1)^{i+1} \varphi(y_0, 
\ldots \hat y_i, \ldots y_r)$ for each $i$ and $X(e) = 0$ everywhere else. Then since $\{y_0, \ldots \hat y_i, \ldots y_r)$ is a basis if and only if $y_i \in C$ we have $\underline X = C$. Now suppose that we have $e, f \in E$. We must show that for every basis $B=\{ f,x_2,\ldots,x_r \}$ with $C \backslash e \subseteq B$, we have
\[
\frac{X(f)}{X(e)} = -\frac{\varphi(e,x_2,\ldots,x_r)}{\varphi(f,x_2,\ldots,x_r)}.
\]
But by Lemma \ref{lem:blergh} it is enough to show this for a single such basis. If $f$ is not in $C$ then both sides are zero, so we may suppose that $f$ is in $C$. Say $e = y_i$ and $f = y_j$. Then taking $(x_2 \ldots x_r)$ to be $(y_0, \ldots \hat y_i, \ldots \hat y_j, \ldots y_r)$ the equation is clear from the definitions. 

The only other nontrivial thing to check is (DP3)$'$ (resp. (DP3)).  To see this, let $X \in \cC_\varphi$ and $Y \in \cD_\varphi$, assuming furthermore that 
$|\underline X \cap \underline Y| \leq 3$ if $\varphi$ is not a strong Grassmann-Pl\"ucker function.
If $\underline{X} \cap \underline{Y} = \emptyset$ then $X \perp Y$ by definition.  Otherwise, we can write $\underline{X} = \{ x_1,\ldots,x_k \}$ and $\underline{Y}=\{ y_1,\ldots,y_\ell \}$ with the elements of $\underline{X} \cap \underline{Y} = \{ x_1,\ldots,x_n \} = \{ y_1,\ldots y_n \}$ written first, so that
$n \geq 1$ and $x_i = y_i$ for $1 \leq i \leq n$.

Since $\underline{X} \backslash x_i$ is independent for all $i=1,\ldots,k$, we must have $k \leq r+1$, and similarly $\ell \leq m-r+1$.
Since $\underline X \setminus x_1$  is independent and $\underline Y \setminus \underline X$ is coindependent in the matroid $\underline M_\varphi$, we can extend $\underline X \setminus x_1$ to a base $B = \{x_2, \ldots x_{r+1}\}$ of $\underline M_\varphi$ disjoint from $\underline Y \setminus \underline X$. Similarly, since $\underline Y \setminus y_1$ is independent and $B - \underline Y$ is coindependent in the matroid $\underline M_{\varphi^*}$, we can extend $Y \setminus y_1$ to a basis $B^*$ of $M_{\varphi^*}$ which is disjoint from $B - \underline Y$. Write $E \backslash (B^* \cup y_1) = \{ z_1,\ldots,z_{r-1} \}$.
If $|\underline X \cap \underline Y| \leq 3$ then $|\{x_1, \ldots, x_{r+1} \} \setminus \{ z_1,\ldots,z_{r-1} \}| =|\underline X \cap \underline Y| \leq 3$.
By either (GP3) or (GP3)$'$, we have
\begin{equation}
\label{eq:Prop4.3a}
\begin{aligned}
&\sum_{i=1}^{r+1} (-1)^i \varphi(x_1,\ldots,\hat{x}_i,\ldots,x_{r+1}) \varphi(x_i,z_1,\ldots,z_{r-1}) \\
&=  \sum_{i=1}^{n} (-1)^i \varphi(x_1,\ldots,\hat{x}_i,\ldots,x_{r+1}) \varphi(x_i,z_1,\ldots,z_{r-1}) \\
&=  \sum_{i=1}^{n} \sigma \cdot \varphi(x_1,\ldots,\hat{x}_i,\ldots,x_{r+1}) \noinvolution{\varphi^*(y_1,\ldots,\hat{y}_i,\ldots,y_{m-r+1})} \in N_G, \\
% 0 &\in \bigboxplus_{i=1}^{r+1} (-1)^i \odot \varphi(x_1,\ldots,\hat{x}_i,\ldots,x_{r+1}) \odot \varphi(x_i,z_1,\ldots,z_{r-1}) \\
% &=  \bigboxplus_{i=1}^{n} (-1)^i \odot \varphi(x_1,\ldots,\hat{x}_i,\ldots,x_{r+1}) \odot \varphi(x_i,z_1,\ldots,z_{r-1}) \\
% &=  \bigboxplus_{i=1}^{n} \sigma \odot \varphi(x_1,\ldots,\hat{x}_i,\ldots,x_{r+1}) \odot \noinvolution{\varphi^*(y_1,\ldots,\hat{y}_i,\ldots,y_{m-r+1})}, \\
\end{aligned}
\end{equation}
where 
\[
\sigma = (-1)^{r-1} {\rm sign}(z_1,\ldots,z_{r-1},y_1,\ldots,y_{m-r+1}).
\]

Multiplying both sides of (\ref{eq:Prop4.3a}) 
by $\sigma \cdot \varphi(x_2,\ldots,x_{r+1})^{-1}\noinvolution{\varphi^*(y_2,\ldots,y_{m-r+1})}^{-1}$
% on the left by $\sigma \cdot \varphi(x_2,\ldots,x_{r+1})^{-1}$ and on the right by
% $\noinvolution{\varphi^*(y_2,\ldots,y_{m-r+1})}^{-1}$
gives
\begin{equation}
\label{eq:Prop4.3b}
\sum_{i=1}^{n} X(x_1)^{-1} X(x_i)  \noinvolution{Y(x_i)} \noinvolution{Y(y_1)}^{-1} \in N_G.
% 0 \in \bigboxplus_{i=1}^{n} X(x_i) \odot X(x_1)^{-1} \odot \noinvolution{Y(x_i)} \odot \noinvolution{Y(y_1)}^{-1}.
\end{equation}

Multiplying both sides of (\ref{eq:Prop4.3b}) by $X(x_1)\noinvolution{Y(y_1)}$ then shows that $X \perp Y$.
% Multiplying both sides of (\ref{eq:Prop4.3b}) on the left by $X(x_1)$ and on the right by $\noinvolution{Y(y_1)}$ then shows that $X \perp Y$.
\end{proof}

\begin{cor}
\label{cor:Cor4.4}
With notation as in Lemma~\ref{lem:Prop4.3}, we have:
\begin{enumerate}
\item For $X \in \cC_\varphi$ and $x_i,x_j \in \underline{X}$,
\[
\frac{X(x_i)}{X(x_j)} = (-1)^{i-j} \frac{\varphi(x_1,\ldots,\hat{x}_i,\ldots,x_{r+1})}{\varphi(x_1,\ldots,\hat{x}_j,\ldots,x_{r+1})}.
\]
\item For $Y \in \cD_\varphi$ and $y_i, y_j \in \underline{Y}$,
\[
\noinvolution{Y(y_j)} \noinvolution{Y(y_i)}^{-1} = \varphi(y_j,z_1,\ldots,z_{r-1}) \varphi(y_i,z_1,\ldots,z_{r-1})^{-1}.
\]
\end{enumerate}
\end{cor}

\begin{proof}
This follows from the same argument as \cite[Corollary 4.4]{AndersonDelucchi}.
% Same as the proof of  \cite[Corollary 4.4]{AndersonDelucchi}, except now we have $Z(y_i) Y(y_i) = -Z(y_j) Y(y_j)$ instead of $Z(y_j)Y(y_i) = -Z(y_i)Y(y_j)$.
\end{proof}

\subsection{Grassmann-Pl{\"u}cker functions from Dual Pairs}

In the previous section, we associated a (weak) dual pair $(\cC_\varphi,\cD_\varphi)$, depending only on the equivalence class of $\varphi$, to each (weak) Grassmann-Pl{\"u}cker function $\varphi$.
However, we don't yet know that $\cC_\varphi$ and $\cD_\varphi$ satisfy the modular elimination axiom (although this will turn out later to be the case).
In this section, we go the other direction, associating a (weak) Grassmann-Pl{\"u}cker function to a (weak) dual pair.

\begin{theorem}
\label{thm:Prop4.6}
Let $\cC$ and $\cD$ be a weak dual pair of $F$-signatures of a matroid $\underline{M}$ of rank $r$.
Then $\cC = \cC_\varphi$ and $\cD = \cD_\varphi$ for a rank $r$ weak Grassmann-Pl{\"u}cker function $\varphi$ which is
uniquely determined up to equivalence. If $\cC$ and $\cD$ form a dual pair then $\varphi$ is a Grassmann-Pl\"ucker function.
\end{theorem}

\begin{proof}
The proof of this result, while rather long and technical, is essentially the same as the special case of phased matroids given in \cite[Proposition 3.6]{AndersonDelucchi}.
Rather than reproduce the entire argument, which takes up 4.5 pages of \cite{AndersonDelucchi}, we will content ourselves with indicating the (minor) changes which need to be made in the present context.

Steps 1 and 2 from {\em loc.~cit.} go through without modification. In Step 1, the correct ratios of the values of $\varphi$ between pairs of bases differing in just 2 elements are calculated, and it is shown that these ratios are consistent with one another. In Step 2, these ratios are used to define the function $\varphi$, and it is shown that $\cC = \cC_{\varphi}$ and $\cD = \cD_{\varphi}$.

% , the orthogonality relation $\cC \perp \cD$ now implies that $\noinvolution{Y(f)} \noinvolution{Y(e)}^{-1} = -X(f)^{-1} X(e)$ rather than $Y(e) Y(f)^{-1} = -X(e)X(f)^{-1}$.
%Thus the displayed equation (4) needs to be replaced with 
%\[
%\noinvolution{Y(e)} \cdot \noinvolution{Y(f)}^{-1} = 
%\varphi_{\cC}(f,t_2,\ldots,t_r)^{-1} \varphi_{\cC}(e,t_2,\ldots,t_r) 
%\]
%(instead of the reciprocal of the right-hand side).

In Step 3, equations (3) and (4) and the assumption $X \perp Y$ show (with notation from {\em loc.~cit.}) that
\begin{equation}
\label{eq:Prop4.6}
\begin{aligned}
&\sum_{x_i \in C_S \cap D_T} X(x_i) \noinvolution{Y(x_i)} \\
&= \sum_{x_i \in C_S \cap D_T} X(x_0)^{-1} X(x_i) Y(x_i) Y(x_0)^{-1} \\
&= \sum_{x_i \in C_S \cap D_T} (-1)^i \varphi_{\cC}(x_1,\ldots,x_r)^{-1} \varphi_{\cC}(x_0,\ldots,\hat{x}_i,\ldots,x_r)\varphi_{\cC}(x_i,y_2,\ldots,y_r) \varphi_{\cC}(x_0,y_2,\ldots,y_r)^{-1} \in N_G,
% &\sum_{x_i \in C_S \cap D_T} X(x_i) \noinvolution{Y(x_i)} \\
% &= \sum_{x_i \in C_S \cap D_T} \frac{X(x_i)}{X(x_0)}  \frac{\noinvolution{Y(x_i)}}{\noinvolution{Y(x_0)}} \\
% &= \sum_{x_i \in C_S \cap D_T} (-1)^i \frac{\varphi_{\cC}(x_0,\ldots,\hat{x}_i,\ldots,x_r)}{\varphi_{\cC}(x_1,\ldots,x_r)} \cdot \frac{\varphi_{\cC}(x_i,y_2,\ldots,y_r)}{\varphi_{\cC}(x_0,y_2,\ldots,y_r)} \in N_G,
\end{aligned}
\end{equation}
and multiplying both sides of (\ref{eq:Prop4.6}) by 
$\varphi_{\cC}(x_1,\ldots,x_r) \varphi_{\cC}(x_0,y_2,\ldots,y_r)$ gives 
% $\varphi_{\cC}(x_1,\ldots,x_r)$ on the left and $\varphi_{\cC}(x_0,y_2,\ldots,y_r)$ on the right gives 
\[
\sum_{x_i \in C_S \cap D_T} (-1)^i  \varphi_{\cC}(x_0,\ldots,\hat{x}_i,\ldots,x_r) \varphi_{\cC}(x_i,y_2,\ldots,y_r) \in N_G, 
\]
which is (GP3).

The proof of (GP3)$'$ from (DP3)$'$ is given by the same calculation, but applied only in cases where the sums in question have at most 3 nonzero summands.
\end{proof}

\subsection{From Grassmann-Pl{\"u}cker functions to Circuits}

In this section, we prove that the set $\cC_\varphi$ of elements of $F^E$ induced by a (weak) Grassmann-Pl{\"u}cker function $\varphi$ is the set of $F$-circuits of a (weak) $F$-matroid with support $\underline M_{\varphi}$.
The only non-trivial axiom is the Modular Elimination axiom (C3)$'$ (resp. (C3)).  

\begin{theorem}
\label{thm:Prop5.3}
{
Let $\varphi$ be a strong (resp. weak) Grassmann-Pl{\"u}cker function on $E$.  Then the set $\cC_\varphi \subseteq F^E$ satisfies the strong Modular Elimination axiom ${\rm (C3)}$  (resp. the weak Modular Elimination axiom ${\rm (C3)}'$).
}
\end{theorem}

\begin{proof}
We prove the strong case first.
Let $\underline M$ be the matroid on $E$ corresponding to the support of $\varphi$. Suppose we have a modular family $X, X_1, \ldots X_k$ and elements $e_1 \ldots e_k \in E$ as in ${\rm(C3)}$. Let $z$ be any element of $\underline X \setminus \bigcup_{i = 1}^k \underline X_i$. Let $A = \underline X \cup \bigcup_{i = 1}^k \underline X_i$, and consider the matroid $N = \underline M |A$.  Since $A$ has height $k + 1$ in the lattice of unions of circuits of $\underline M$, the rank of $N^*$ is $k + 1$. Thus the rank of $N$ is $|A| - k - 1$. The set $I = A \setminus \{z, e_1, \ldots e_k\}$ has this rank and is spanning, so it is a basis of $N$. Let $Z \in \cC_{\varphi}$ with $\underline Z$ given by the fundamental circuit of $z$ with respect to $I$ and with $Z(z) = X(z)$. It is clear that $Z(e_i) = 0$ for $1 \leq i \leq k$. We must show that for any $f \in E$ we have $-Z(f) + X(f) + \sum_{i=1}^k X_i(f) \in N_G$. This is clear if $f$ is $z$ or one of the $e_i$ or if $f \not \in A$, so we may suppose that $f \in I$.

Let $J$ be a basis of $N$ including $\underline X \setminus\{ z\}$ and let $K$ be a basis of $M/A$. Then $B_1 = J \dot \cup K$ and $B_2 = I \dot \cup K$ are bases of $M$. Let $x_1 = z$ and let $x_2, \ldots, x_{r+1}$ enumerate $B_1$. Let $y_1, \ldots, y_{r-1}$ enumerate $B_2 \setminus \{f\}$. We define the constants $\lambda_1$ and $\lambda_2$ by
$$\lambda_1 = \varphi(x_2, \ldots x_{r+1})X(z)^{-1} \qquad \lambda_2 = \varphi(f, y_1, \ldots, y_{r-1}).$$
% $$\lambda_1 = \frac{\varphi(x_2, \ldots x_{r+1})}{X(z)} \qquad \lambda_2 = \varphi(f, y_1, \ldots, y_{r-1}).$$
% {\bf Matt: Careful with non-commutativity!}
Consider any $i$ with $2 \leq i \leq r$. If $x_i \not \in \underline X$ then $\{x_1, \ldots, \hat x_i, \ldots x_{r+1}\}$ is not a basis, so $\varphi(x_1, \ldots, \hat x_i, \ldots x_{r+1}) = 0$. If $x_i \in \underline X$ then $$\frac{X(x_i)}{X(z)} = -\frac{\varphi(z, x_2, \ldots, \hat x_i, \ldots x_{r+1})}{\varphi(x_i, x_2, \ldots  \hat x_i, \ldots x_{r+1})}$$ and in either case it follows that $$\varphi(x_1, \ldots, \hat x_i, \ldots x_{r+1}) = (-1)^i \lambda_1 X(x_i).$$
This formula also clearly holds for $i = 1$.

For $1 \leq i \leq k$, if $f \not \in \underline X_i$ then $\{e_i, y_1, \ldots y_{r-1}\}$ is not a basis of $\underline M$, so $\varphi(e_i, y_1, \ldots y_{r-1}) = 0$. If $f \in \underline X_i$ then we have $$\frac{X_i(f)}{X_i(e_i)} = - \frac{\varphi(e_i, y_1, \ldots, y_{r-1})}{\varphi(f, y_1, \ldots, y_{r-1})}.$$ In either case, it follows that $$\varphi(e_i, y_1, \ldots, y_{r-1}) = \lambda_2 \frac{X_i(f)}{X(e_i)}.$$ Similarly we have $$\varphi(z, y_1, \ldots, y_{r-1}) = -\lambda_2\frac{Z(f)}{X(z)}.$$

Applying (GP3) we have
$$\sum_{s = 1}^{r+1} (-1)^s \varphi(x_1, \ldots, \hat x_s, \ldots, x_{r+1}) \varphi(x_s, y_1, \ldots y_{r-1}) \in N_G.$$
Many of these summands are 0. If $x_s \not \in A$ then $\varphi(x_1, \ldots, \hat x_s, \ldots, x_{r+1}) = 0$. If $x_s \in I \setminus \{f\}$ then $\varphi(x_s, y_1, \ldots, y_{r-1}) = 0$. The only other possibilities are $x_s = z$, $x_s = f$, or $x_s = e_i$ for some $i$. 
So we have
% \footnote{{\bf Matt:} Actually I don't think this identity holds in the non-commutative case\ldots}
\[
% \begin{aligned}
-\lambda_1\lambda_2Z(f) + \lambda_1\lambda_2X(f) + \sum_{i = 1}^k\lambda_1\lambda_2 X_i(f) \\
= \lambda_1\lambda_2 \left(-Z(f) + X(f) + \sum_{i = 1}^k X_i(f)  \right) \in N_G, \\
% \end{aligned}
\]
from which it follows that $- Z(f) + X(f) +  \sum_{i=1}^k X_i(f) \in N_G$.

The proof for weak Grassman-Pl{\"u}cker functions is essentially the same, but in the special case that $k = 1$. This ensures that $|\{x_1, \ldots, x_{r+1}\} \setminus \{y_1, \ldots y_{r-1}\}| = |\{z, f, e_1\}| \leq 3$, so that (GP3)$'$ can be applied instead of (GP3).
\end{proof}

\subsection{From Circuits to Dual Pairs}

We begin with the following result giving a weak version of the modular elimination axiom which holds for pairs of $F$-circuits that are not necessarily modular.

\begin{lemma}
\label{lem:Lemma5.4}
Let $\cC$ be the set of $F$-circuits of a weak $F$-matroid $M$.  Then for all $X,Y \in \cC$, $e,f \in E$ with $X(e)=-Y(e) \neq 0$ and $Y(f) \neq -X(f)$, there is $Z \in \cC$ with 
$f \in \underline{Z} \subseteq (\underline{X} \cup \underline{Y}) \backslash e$.
\end{lemma}

\begin{proof}
This follows from the proof of \cite[Lemma 5.4]{AndersonDelucchi}, where $X'(g) \leq X(g)$ in {\em loc.~cit.} is interpreted to mean that $X'(g)=0$ or $X'(g)=X(g)$ (and similarly for $Y'(g)$ and $Y(g)$).  
Note that the proof of \cite[Proposition 5.1]{AndersonDelucchi}, which is used in the proof of Lemma 5.4 of {\em loc.~cit.}, holds {\em mutatis mutandis} for weak matroids over a hyperfield $F$.
\end{proof}

The proof of the following result diverges somewhat from the treatment of the analogous assertion in \cite{AndersonDelucchi}.

\begin{theorem}
\label{thm:Prop5.6}
Let $\cC$ be the $F$-circuit set of a weak $F$-matroid $M$.  There is a unique $F$-signature $\cD$ of $\underline{M}^*$  such that $(\cC, \cD)$ form a weak dual pair of
$F$-signatures of $\underline{M}$. 
If $M$ is a strong $F$-matroid then $(\cC, \cD)$ form a dual pair.
\end{theorem}

\begin{proof}
Let $D$ be a cocircuit of $\underline{M}$.  As in the proof of \cite[Proposition 5.6]{AndersonDelucchi}, choose a maximal independent subset $A$ of $D^c$.  For $e,f \in D$, choose
$X_{D,e,f} \in \cC$ with support equal to the unique circuit $C_{D,e,f}$ of $\underline{M}$ with support contained in $A \cup \{ e,f \}$.  Define $\cD$ to be the collection of all
$W \in F^E$ with support some cocircuit $D$ such that 
\begin{equation}
\label{eq:Prop5.6}
\frac{W(e)}{W(f)} = -\frac{\noinvolution{X_{D,e,f}(f)}}{\noinvolution{X_{D,e,f}(e)}}
\end{equation}
for all $e,f \in D$.

By the proof of Claim 1 in \cite[Proof of Proposition 5.6]{AndersonDelucchi}, 
the set $\cD$ is well-defined and independent of the choice of $X_{D,e,f}$.

It remains to prove (DP3) (resp. (DP3)$'$). Let $X \in \cC$ and $Y \in \cD$, and if $M$ is a weak but not a strong $F$-matroid assume furthermore that $|\underline X \cap \underline Y| \leq 3$. If $\underline X \cap \underline Y$ is empty then we are done, so suppose that $\underline X \cap \underline Y$ is nonempty.
Since $\underline M$ is a matroid, $\underline X \cap \underline Y$ must contain at least two elements, so let $\underline X \cap \underline Y = \{z, e_1 \ldots e_k\}$ with $k \geq 1$. We may assume without loss of generality that $Y(z) = 1$. Let $I$ be a basis of $\underline M \backslash \underline Y$ including $\underline X \setminus \underline Y$. Then $B = I \cup \{z\}$ is a basis of $\underline M$. For $1 \leq i \leq k-1$ let $X_i \in \Ccal$ with $\underline X_i$ the fundamental circuit of $e_i$ with respect to $B$ and $X_i(e_i) = -X(e_i)$. Let $C$ be the fundamental circuit of $e_k$ with respect to $B$. 

We have $\underline X \setminus B \subseteq \{e_1, \ldots, e_k\}$, and for any $e \in \underline X \cap B$ the fundamental cocircuit of $e$ with respect to $B$ must meet $\underline X$ again, and must do so in some element of $\underline X \setminus B$. Thus $\underline X \subseteq C \cup \bigcup_{i = 1}^{k-1} \underline X_i$, which has height $k$ in the lattice of unions of circuits of $\underline M$. It follows that $X$ and the $X_i$ form a modular family (resp. a modular pair). So there is some $Z \in \cC$ with $Z(e_i) = 0$ for $1 \leq i \leq k-1$ and $-Z(f) + X(f) + \sum_{i = 1}^{k-1}X_i(f) \in N_G$ for any $f \in E$. Applying this with $f = e_k$ gives $Z(e_k) = X(e_k)$. For $1 \leq i \leq k-1$ we have $$\noinvolution{Y(e_i)} = \frac{\noinvolution{Y(e_i)}}{\noinvolution{Y(z)}} = -
\frac{X_i(z)}{X_i(e_i)},$$
so that $X_i(z) = -X_i(e_i)\noinvolution{Y(e_i)} = X(e_i)\noinvolution{Y(e_i)}$. Similarly, we have $Z(z) = -Z(e_k)\noinvolution{Y(e_k)} = -X(e_k)\noinvolution{Y(e_k)}$. This gives 
\[
\begin{aligned}
X \cdot Y &= X(e_k)Y(e_k) + X(z)Y(z) + \sum_{i = 1}^{k-1}X(e_i)Y(e_i) \\
&= -Z(z) + X(z) + \sum_{i = 1}^{k-1}X_i(z) \\
&\in N_G.
\end{aligned}
\]
Thus $X \perp Y$.
\end{proof}

\subsection{Cryptomorphic axiom systems for $F$-matroids}

We can finally prove the main theorems
from \S\ref{sec:MatroidsOverTracts}.
We begin by proving Theorems~\ref{thm:A} and \ref{thm:C} together in the following result:

\begin{theorem} \label{thm:maintheorem}
Let $E$ be a finite set.  There are natural bijections between the following three kinds of objects:
\begin{itemize}
\item[(C)] Collections $\cC \subset F^E$ satisfying {\rm (C0),(C1),(C2),(C3)}.
\item[(GP)] Equivalence classes of Grassmann-Pl{\"u}cker functions on $E$ satisfying {\rm (GP1),(GP2),(GP3)}.
\item[(DP)] Matroids $\underline{M}$ on $E$ together with a dual pair $(\cC,\cD)$ satisfying {\rm (DP1),(DP2),(DP3)}.
\end{itemize}
\end{theorem}

\begin{proof}

(GP)$\Rightarrow$(C): If $\varphi$ is a Grassmann-Pl{\"u}cker function, Theorem~\ref{thm:Prop5.3} shows that the set $C_\varphi$ from Definition~\ref{defn:Cphi} satisfies (C0)-(C3).

\medskip

(C)$\Rightarrow$ (DP): If $\cC$ satisfies (C0)-(C3) and $M$ denotes the corresponding $F$-matroid, 
Theorem~\ref{thm:Prop5.6} shows that there is a unique signature $\cD$ of $\underline{M}^*$ such that $(\cC,\cD)$ is a dual pair of $F$-signatures of $\underline{M}$.

\medskip

(DP)$\Rightarrow$(GP): If $(\cC,\cD)$ is a dual pair of $F$-signatures of a rank $r$ matroid $\underline{M}$, Theorem~\ref{thm:Prop4.6} shows that there is a unique equivalence class of Grassmann-Pl{\"u}cker function $\varphi: E^r \to F$ such that $\cC = \cC_{\varphi}$ and $\cD = \cD_{\varphi}$.
\end{proof}

%We also have the following supplement:
Similarly we have:
\begin{theorem} \label{thm:maintheorem'}
Let $E$ be a finite set.  There are natural bijections between the following three kinds of objects:
\begin{itemize}
\item[(C)] Collections $\cC \subset F^E$ satisfying {\rm (C0),(C1),(C2),(C3)}$'$.
\item[(GP)] Equivalence classes of Grassmann-Pl{\"u}cker functions on $E$ satisfying {\rm (GP1),(GP2),(GP3)}$'$.
\item[(DP)] Matroids $\underline{M}$ on $E$ together with a dual pair $(\cC,\cD)$ satisfying {\rm (DP1),(DP2),(DP3)}$'$.
\end{itemize}
\end{theorem}

\subsection{Duality for $F$-matroids}

In this section, we prove Theorems~\ref{thm:B} and \ref{thm:D}.  We begin with the following preliminary result:

\begin{lemma}
\label{lem:Prop5.8}
Let $\cC \subseteq F^E$ be the set of $F$-circuits of a (weak) $F$-matroid $M$.  Then the set of elements of $\cC^\perp \backslash \{ 0 \}$ of minimal 
non-empty support is exactly the signature $\cD$ of $\underline{M}^*$ given by Theorem~\ref{thm:Prop5.6}.
\end{lemma}

\begin{proof}
This is proved exactly like \cite[Proof of Proposition 5.8]{AndersonDelucchi}.
\end{proof}

\begin{proof}[Proof of Theorem~\ref{thm:B}:]
This follows from Theorem~\ref{thm:maintheorem}, Lemma~\ref{lem:Lemma3.2}, and Proposition~\ref{lem:Prop4.3} and \ref{lem:Prop5.8}, exactly as in \cite[Proof of Theorem B]{AndersonDelucchi}.
\end{proof}

\begin{proof}[Proof of Theorem~\ref{thm:D}:]
(cf.~\cite[Proof of Theorem D]{AndersonDelucchi}) This follows from Theorem~\ref{thm:maintheorem} and Lemmas~\ref{lem:MinorLemma} and \ref{lem:Prop4.3}.
\end{proof}

\subsection{Proof of Theorem~\ref{thm:C3primeprime}} \label{sec:C3primeprime}

{ 
In this section, we prove Theorem~\ref{thm:C3primeprime}.

\begin{proof}[Proof of Theorem~\ref{thm:C3primeprime}:]
Suppose first that (C3) holds.  Then in particular ${\rm (C3)}'$ holds, ${\mathcal C}$ is the set of $F$-circuits of a weak $F$-matroid $M$, and
the support of ${\mathcal C}$ is the set of circuits of the underlying matroid $\underline{M}$.
Let $X \in {\mathcal C}$ and let $B$ be a basis of $\underline{M}$.  Write $\underline{X} \backslash B = \{ e_1,\ldots,e_k \}$ and set $e = e_1$.
For $1 \leq i \leq k$ let $X_i = X(e) X_{B,e_i}$.
One checks easily that $X_2,\ldots,X_k$ and $-X$ satisfy the hypotheses of (C3), and thus 
there is an $F$-circuit $Z$ such that $Z(e_i) = 0$ for $1 \leq i \leq k$ and $-Z(f) -X(f) + X_2(f) + \cdots + X_k(f) \in N_G$
for every $f \in E$.

We must have $Z(e) \neq 0$ or else $\underline{Z} \subseteq B$, which is impossible.
As $f \in B$ for all other $f \in \underline{Z}$ and $Z(e) = -X(e)$, we must have $Z = -X(e)X_{B,e} = -X_1$.
Thus $-X(f) + X_1(f) + X_2(f) + \cdots + X_k(f) \in N_G$ for all $f \in E$, establishing ${\rm (C3)}''$.

Now assume that ${\rm (C3)}''$ holds.  
Suppose we have a modular family $X, X_1, \ldots X_k$ and elements $e_1 \ldots e_k \in E$ as in ${\rm(C3)}$.
As in the proof of Theorem~\ref{thm:Prop5.3}, if $z$ is any element of $\underline X \setminus \bigcup_{i = 1}^k \underline X_i$, $A = \underline X \cup \bigcup_{i = 1}^k \underline X_i$, and $N = \underline M |A$, then
$I = A \setminus \{z, e_1, \ldots e_k\}$ is a basis of $N$.
Let $J$ be a basis of $M/A$.  Then $B = I \dot \cup J$ is a basis of $M$, 
and $X_i = -X(e_i) X_{B,e_i}$ for all $i=1,\ldots,k$.

Let $Z \in \cC$ with $\underline Z$ given by the fundamental circuit of $z$ with respect to $I$ and with $Z(z) = X(z)$. It is clear that $Z(e_i) = 0$ for $1 \leq i \leq k$, and it follows by inspection that $-Z(f) +X(f) + \sum_{i=1}^k X_i(f) \in N_G$ for all $f \in E$, establishing (C3).
\end{proof}
}

\subsection{Strong and weak matroids coincide over perfect tracts}\label{sec:perfectproof}
In this section, we prove Theorem \ref{thm:perfect}.
We will need to consider, for each natural number $k$, the following weakening of (DP3):
\begin{itemize}
\item[(DP3)$_{k}$] $X \perp Y$ for every pair $X \in \Ccal$ and $Y \in \Dcal$ with $|\underline X \cap \underline Y| \leq k$.
\end{itemize}
So (DP3)$_3$ is just ${\rm (DP3)}'$, and (DP3) is equivalent to the conjunction of all the (DP3)$_k$.

\begin{proof}[Proof of Theorem \ref{thm:perfect}]
We will show by induction on $k$ that any weak $F$-matroid satisfies (DP3)$_k$ for all $k \geq 3$. The base case $k = 3$ is true by definition. So let $k > 3$ and suppose that every weak $F$-matroid satisfies (DP3)$_{k-1}$. Let $M$ be a weak $F$-matroid, and choose $X \in \Ccal$ and $Y \in \Dcal$ with $|\underline X \cap \underline Y| \leq k$. We must show that $X \perp Y$. This follows from (DP3)$_{k-1}$ if $|\underline X \cap \underline Y| \leq k-1$, so we may suppose that $|\underline X \cap \underline Y| = k$.

By contracting $\underline X \setminus \underline Y$ and deleting $\underline Y \setminus \underline X$ if necessary\footnote{These operations were introduced in Subsection \ref{sec:minors}.}, we may assume without loss of generality that $\underline X = \underline Y$. By contracting a basis of 
{$\underline{M}/\underline X$} if necessary, 
% this previously said $M/\underline X$}
we may assume without loss of generality that $\underline X$ is spanning in {$\underline{M}$}. Similarly we may assume without loss of generality that $\underline Y$ is cospanning in {$\underline{M}$}. So the rank and the corank of $M$ are both $k-1$, which means that $M$ has $2k-2$ elements. So $M$ has at least $k-2 \geq 2$ elements outside $\underline X$. None of these elements can be coloops (since $\underline X$ is spanning) or loops (since it is cospanning). Let $N$ be a minor of $M$ with ground set $\underline X$, in which at least one of the edges outside $\underline X$ has been contracted and at least one has been deleted. This ensures that $\underline X$ is not a circuit of $\underline N$, and dually it also ensures that $\underline X$ is not a cocircuit of $\underline N$.

For any cocircuit $W$ of $N$ there is some cocircuit $\hat W$ of $M$ with $W = \hat W \restric_{\underline X}$. Then $|\underline X \cap \underline{\hat W}| = |\underline W| \leq k-1$, so $X \perp \hat W$, from which it follows that $X \restric_{\underline X} \perp W$. So $X \restric_{\underline X}$ is a vector of $N$. Similarly $Y\restric_{\underline X}$ is a covector of $N$. Any intersection of a circuit with a cocircuit of $\underline N$ has at most $k-1$ elements. Since $N$ is a weak $F$-matroid, it satisfies (DP3)$_{k-1}$ by the induction hypothesis, and so it is in fact a strong $F$-matroid. Since $F$ is perfect, it follows that $X \restric_{\underline X} \perp Y \restric_{\underline X}$ and so $X \perp Y$, as required.

We have now shown that any weak matroid over $F$ satisfies (DP3)$_k$ for all $k$, and so is strong.
\end{proof}

\appendix

\section{Errata to \cite{AndersonDelucchi}}
\label{sec:errata}

Since we rely so heavily in this paper on \cite{AndersonDelucchi}, we include the following list of errata.

\medskip

Most of the errors in \cite{AndersonDelucchi} are minor and localized, but there is one major problem which affects the paper globally.
(A similar error is present in the arXiv versions 1 through 3 of the present paper.)
The difficulty is in the third paragraph of the proof of Claim 3 on page 831.
The authors write that if $X$ is not orthogonal to $W$ then neither is $X'$. But in order for that conclusion to follow, 
one would need to know that $X'$ agrees with $X$ on the domain of $X'$.  However, there is no reason to expect this to hold.
Indeed, as Example~\ref{eg:danscex} shows, Theorem A in \cite{AndersonDelucchi} does not hold. 
\medskip

In addition, we mention the following less serious mistakes:

\begin{enumerate}
\item In Definition 2.4, there should be an additional axiom that the zero vector is not a phased circuit.  And axiom (C1) should say ${\rm supp}(X) \subseteq {\rm supp}(Y)$ rather than ${\rm supp}(X) = {\rm supp}(Y)$.
\item In the proof of Lemma 3.2, $E \backslash (X \cap Y)$ should be $E \backslash (X \cup Y)$.
\item In the first bulleted point of \S{4.2} (top of page 822), $b_0$ should be $b_1$.
\item In the statement of Lemma 5.2, $X(e)=Y(e)$ should be $X(e)=-Y(e)$ and ${\mathcal C}$ should be ${\mathcal C}_\varphi$.   Note that Lemma 5.2 is not actually used in any of the subsequent arguments.
\item In the statements of Proposition 5.3 and Lemma 5.4, the hypothesis $X(f) \neq Y(f)$ should be replaced with $X(f) \neq -Y(f)$.  And in the third line from the end of the proof of Lemma 5.4, $X(f) \neq Y(f) = Y'(f)$ should be $-X(f) \neq Y(f) = Y'(f)$. 
\item In Lemmas 4.5 and Proposition 5.6, the correct hypotheses are that ${\mathcal C}$ and ${\mathcal D}$ form a dual pair of circuit signatures for some matroid $M$.  This is all that is used in the proofs, and if one makes the stronger assumption in Proposition 5.6 that ${\mathcal C},{\mathcal D}$ are the phased circuits (resp.~cocircuits)
of a phased matroid then the proof of Corollary 5.7 is incomplete. 
\item In the second line of the proof of Proposition 5.3, the authors refer to the cocircuits of the phased matroid defined by $\varphi$, but one doesn't actually know at this point in their chain of reasoning that the modular elimination axiom holds for what eventually ends up being the phased matroid defined by $\varphi$.  Their proof is nevertheless correct.
% \item I was not able to understand the proof of Claim 3 in (the proof) of Proposition 5.6. Claims 2 and 3 can, however, be combined and proved in a single, shorter, argument: see the proof of Theorem~\ref{thm:Prop5.6} above.
\end{enumerate}

\begin{remark}
In Definition 2.4, the authors write $Z(g) \leq \max\{X(g),Y(g)\}$ in the ``else'' case, but this inequality can be replaced with equality; this follows from the ``symmetric difference'' part of \cite[Lemma 2.7.1]{WhiteCG}.
The latter result also implies that axiom (ME) in Definition 2.4 (and also in Proposition A.21) can be replaced with a stronger axiom in which one asks for a {\bf unique} $Z \in {\mathcal C}$ with the stated properties.
\end{remark}

\section{Fuzzy rings simplified (written by Oliver Lorscheid)}
\label{sec:Lorscheid}

In this appendix, we show that every fuzzy ring is weakly isomorphic to a fuzzy ring of a particularly simple form. To be more precise, we describe a full subcategory $\FuzzGrRings$ of the category $\Fuzz$ of fuzzy rings together with weak morphisms that is equivalent to $\Fuzz$.

The objects of $\FuzzGrRings$ are defined as the fuzzy rings $(K;+;\cdot;\epsilon;K_0)$ for which the triple $(K,+,\cdot)$ is a semiring that is isomorphic to the group semiring $\N[G]$ of an abelian group $G$.
%  whose elements are finite formal sum $\sum n_x x$ of elements $x\in G$. 

Note that if $(K,+,\cdot)$ is a group semiring, then the axioms (FR0), (FR1), (FR2) and (FR7) of a fuzzy ring are automatically satisfied and axiom (FR3) is equivalent to the fact that $\epsilon\in G$. Note further that $G=K^\ast$. 

Thus $\FuzzGrRings$ consists of quintuples $(K;+;\cdot;\epsilon;K_0)$ for which $(K,+,\cdot)$ is a commutative semiring equal to $\N[K^\ast]$ and such that $\epsilon\in K^\ast$ and $K_0\subseteq K$ satisfy the following axioms:
\begin{enumerate}
 \item[(FR4)] $K_0$ is a proper semiring ideal, i.e.\ $K_0+K_0\subseteq K_0$, $K\cdot K_0\subseteq K_0$, $0\in K_0$ and $1\notin K_0$.
 \item[(FR5)] For $\alpha\in K^\ast$, we have $1+\alpha\in K_0$ if and only if $\alpha=\epsilon$.
 \item[(FR6)] If $x_1,x_2,y_1,y_2\in K$ and $x_1+y_1,x_2+y_2\in K_0$, then $x_1\cdot x_2+\epsilon \cdot y_1\cdot y_2\in K_0$.
\end{enumerate}

\begin{proposition}
 The inclusion functor $\FuzzGrRings\to\Fuzz$ is an equivalence of categories. In particular, a fuzzy ring $(K;+;\cdot;\epsilon;K_0)$ is weakly isomorphic to the fuzzy ring $(K';+;\cdot;\epsilon';K_0')$ that is defined as follows: 
 \begin{itemize}
  \item $(K',+,\cdot)=\N[K^\ast]$ as semirings;
  \item $\epsilon'=1\cdot \epsilon$, considered as an element of $K^\ast\subseteq\N[K^\ast]=K'$;
  \item $K_0'=\bigl\{ \,\sum n_xx \in\N[K^\ast] \, \bigl| \,\sum n_xx \in K_0\text{ as an element of }K\, \bigr\}$.
 \end{itemize}
\end{proposition}

\begin{proof}
 It is clear that the inclusion functor $\FuzzGrRings\to\Fuzz$ is fully faithful. Thus it suffices to show that this functor is essentially surjective. This follows from the latter claim of the proposition.
 
 To begin with, we reason that the quintuple $(K';+;\cdot;\epsilon';K_0')$ is indeed a fuzzy ring. As observed before, it is enough to verify axioms (FR4), (FR5) and (FR6). Axioms (FR4) and (FR6) follow immediately from the corresponding properties for $K$. Axiom (FR5) follows from the corresponding property for $K$ and the fact that $(K')^\ast=(\N[K^\ast])^\ast=K^\ast$. 
 
 In what follows, we show that the identity map $f:(K')^\ast\to K^\ast$, with respect to the identification $(K')^\ast=K^\ast$, defines a weak isomorphism $K'\to K$ of fuzzy rings. 
 
 To begin with, we verify that $f$ is a weak morphism. If $\sum n_x x \in K_0'$, then $\sum n_x x\in K_0$ by the very definition of $K_0'$. Thus $f$ is a weak morphism.
 
 We continue with the verification that the identity map $g:K^\ast\to (K')^\ast$ defines a weak morphism $K\to K'$. Consider a sum $\sum x_i$ of elements $x_1,\dotsc,x_n\in K^\ast$ that is contained in $K_0$. The corresponding element of $K'$ is $\sum n_x x$, where $n_x$ equals the number of indices $i$ between $1$ and $n$ for which $x_i=x$. Again by the definition of $K_0'$, this is an element of $K_0'$. This shows that $g$ defines a weak morphism $K\to K'$.
 
 Since $f$ and $g$ are mutual inverse maps, the corresponding weak morphisms between $K$ and $K'$ are mutual inverse weak isomorphisms, which completes the proof of the proposition.
\end{proof}

%%%%%%%%%%%%%%%%%%%%%%%%%%%%%%%%%%%%%%%%%%%%%%%%%%%%%%%%%%%%%%%%%%%%%%

\bibliographystyle{alpha}
\bibliography{Hyperfield}

\newcommand{\etalchar}[1]{$^{#1}$}
\def\cprime{$'$} \def\cprime{$'$} \def\cprime{$'$}
\begin{thebibliography}{BLVS{\etalchar{+}}99}

\bibitem[AD12]{AndersonDelucchi}
Laura Anderson and Emanuele Delucchi.
\newblock Foundations for a theory of complex matroids.
\newblock {\em Discrete Comput. Geom.}, 48(4):807--846, 2012.

\bibitem[And19]{AndersonVectors}
Laura Anderson.
\newblock Vectors of matroids over tracts.
\newblock {\em J. Combin. Theory Ser. A}, 161:236--270, 2019.

\bibitem[BB17]{BakerBowlerHyperfield}
Matthew Baker and Nathan Bowler.
\newblock Matroids over hyperfields.
\newblock Preprint. Available at {\tt arxiv:math.CO/1601.01204}, 31 pages,
  2017.

\bibitem[Ber90]{BerkovichBook}
Vladimir~G. Berkovich.
\newblock {\em Spectral theory and analytic geometry over non-{Archimedean}
  fields}, volume~33 of {\em Mathematical Surveys and Monographs}.
\newblock American Mathematical Society, Providence, RI, 1990.

\bibitem[BJ87]{BlandJensen}
Robert~G. Bland and David~L. Jensen.
\newblock Weakly oriented matroids.
\newblock Cornell University School of OR/IE Technical Report No. 732, 1987.

\bibitem[BLV78]{BlandLasVergnas}
Robert~G. Bland and Michel Las~Vergnas.
\newblock Orientability of matroids.
\newblock {\em J. Combinatorial Theory Ser. B.}, 24(1):94--123, 1978.

\bibitem[BLVS{\etalchar{+}}99]{OM}
Anders Bj{\"o}rner, Michel Las~Vergnas, Bernd Sturmfels, Neil White, and
  G{\"u}nter~M. Ziegler.
\newblock {\em Oriented matroids}, volume~46 of {\em Encyclopedia of
  Mathematics and its Applications}.
\newblock Cambridge University Press, Cambridge, second edition, 1999.

\bibitem[BPR06]{BasuPollackRoy}
Saugata Basu, Richard Pollack, and Marie-Fran{\c{c}}oise Roy.
\newblock {\em Algorithms in real algebraic geometry}, volume~10 of {\em
  Algorithms and Computation in Mathematics}.
\newblock Springer-Verlag, Berlin, second edition, 2006.

\bibitem[CC10]{ConnesConsaniAbsolute}
Alain Connes and Caterina Consani.
\newblock From monoids to hyperstructures: in search of an absolute arithmetic.
\newblock In {\em Casimir force, {C}asimir operators and the {R}iemann
  hypothesis}, pages 147--198. Walter de Gruyter, Berlin, 2010.

\bibitem[CC11]{ConnesConsani}
Alain Connes and Caterina Consani.
\newblock The hyperring of ad\`ele classes.
\newblock {\em J. Number Theory}, 131(2):159--194, 2011.

\bibitem[Del11]{Delucchi}
Emanuele Delucchi.
\newblock Modular elimination in matroids and oriented matroids.
\newblock {\em European J. Combin.}, 32(3):339--343, 2011.

\bibitem[Dre86]{Dress}
Andreas W.~M. Dress.
\newblock Duality theory for finite and infinite matroids with coefficients.
\newblock {\em Adv. in Math.}, 59(2):97--123, 1986.

\bibitem[DW91]{DressWenzelGP}
Andreas W.~M. Dress and Walter Wenzel.
\newblock Grassmann-{P}l\"ucker relations and matroids with coefficients.
\newblock {\em Adv. Math.}, 86(1):68--110, 1991.

\bibitem[DW92a]{DressWenzelVM}
Andreas W.~M. Dress and Walter Wenzel.
\newblock Valuated matroids.
\newblock {\em Adv. Math.}, 93(2):214--250, 1992.

\bibitem[DW92b]{DressWenzelPM}
Andreas~W.M. Dress and Walter Wenzel.
\newblock Perfect matroids.
\newblock {\em Advances in Mathematics}, 91(2):158 -- 208, 1992.

\bibitem[FM16]{FinkMoci}
Alex Fink and Luca Moci.
\newblock Matroids over a ring.
\newblock {\em J. Eur. Math. Soc. (JEMS)}, 18(4):681--731, 2016.

\bibitem[Fre13]{Frenk}
Bart Frenk.
\newblock Tropical varieties, maps, and gossip.
\newblock Ph.D. thesis. Available at {\tt
  http://alexandria.tue.nl/extra2/750815.pdf}, 167 pages, 2013.

\bibitem[GG18]{GiansiracusaGrassmann}
Jeffrey Giansiracusa and Noah Giansiracusa.
\newblock A {G}rassmann algebra for matroids.
\newblock {\em Manuscripta Math.}, 156(1-2):187--213, 2018.

\bibitem[IR10]{IRsupertropicalalgebra}
Zur Izhakian and Louis Rowen.
\newblock Supertropical algebra.
\newblock {\em Adv. Math.}, 225(4):2222--2286, 2010.

\bibitem[IR11]{IRmatrixalgebra}
Zur Izhakian and Louis Rowen.
\newblock Supertropical matrix algebra.
\newblock {\em Israel J. Math.}, 182:383--424, 2011.

\bibitem[Jun18]{JunHyperringScheme}
Jaiung Jun.
\newblock Algebraic geometry over hyperrings.
\newblock {\em Adv. Math.}, 323:142--192, 2018.

\bibitem[KL72]{KleimanLaksov}
Steven~L. Kleiman and Dan Laksov.
\newblock Schubert calculus.
\newblock {\em Amer. Math. Monthly}, 79:1061--1082, 1972.

\bibitem[Mar96]{MarshallBook}
Murray~A. Marshall.
\newblock {\em Spaces of orderings and abstract real spectra}, volume 1636 of
  {\em Lecture Notes in Mathematics}.
\newblock Springer-Verlag, Berlin, 1996.

\bibitem[Mar06]{Marshall}
Murray~A. Marshall.
\newblock Real reduced multirings and multifields.
\newblock {\em J. Pure Appl. Algebra}, 205(2):452--468, 2006.

\bibitem[Mas85]{Massouros}
Ch.~G. Massouros.
\newblock Methods of constructing hyperfields.
\newblock {\em Internat. J. Math. Math. Sci.}, 8(4):725--728, 1985.

\bibitem[MS15]{MaclaganSturmfels}
Diane Maclagan and Bernd Sturmfels.
\newblock {\em Introduction to tropical geometry}, volume 161 of {\em Graduate
  Studies in Mathematics}.
\newblock American Mathematical Society, Providence, RI, 2015.

\bibitem[MT01]{MurotaTamura}
Kazuo Murota and Akihisa Tamura.
\newblock On circuit valuation of matroids.
\newblock {\em Adv. in Appl. Math.}, 26(3):192--225, 2001.

\bibitem[Oxl92]{Oxley}
James~G. Oxley.
\newblock {\em Matroid theory}.
\newblock Oxford Science Publications. The Clarendon Press, Oxford University
  Press, New York, 1992.

\bibitem[PvZ10]{PvZLifts}
R.~A. Pendavingh and S.~H.~M. van Zwam.
\newblock Lifts of matroid representations over partial fields.
\newblock {\em J. Combin. Theory Ser. B}, 100(1):36--67, 2010.

\bibitem[PvZ13]{PvZSkew}
R.~A. Pendavingh and S.~H.~M. van Zwam.
\newblock Skew partial fields, multilinear representations of matroids, and a
  matrix tree theorem.
\newblock {\em Adv. in Appl. Math.}, 50(1):201--227, 2013.

\bibitem[SW96]{SempleWhittle}
Charles Semple and Geoff Whittle.
\newblock Partial fields and matroid representation.
\newblock {\em Adv. in Appl. Math.}, 17(2):184--208, 1996.

\bibitem[Tut58]{Tutte}
W.~T. Tutte.
\newblock A homotopy theorem for matroids. {I}, {II}.
\newblock {\em Trans. Amer. Math. Soc.}, 88:144--174, 1958.

\bibitem[Vir10]{ViroDequant}
Oleg~Y. Viro.
\newblock Hyperfields for {T}ropical {G}eometry {I}. {H}yperfields and
  dequantization.
\newblock Preprint. Available at {\tt arxiv:math.AG/1006.3034}, 45 pages, 2010.

\bibitem[Vir11]{Viro}
Oleg~Y. Viro.
\newblock On basic concepts of tropical geometry.
\newblock {\em Tr. Mat. Inst. Steklova}, 273(Sovremennye Problemy
  Matematiki):271--303, 2011.

\bibitem[Whi87]{WhiteCG}
Neil White, editor.
\newblock {\em Combinatorial geometries}, volume~29 of {\em Encyclopedia of
  Mathematics and its Applications}.
\newblock Cambridge University Press, Cambridge, 1987.

\bibitem[Whi97]{Whittle}
Geoff Whittle.
\newblock On matroids representable over {${\rm GF}(3)$} and other fields.
\newblock {\em Trans. Amer. Math. Soc.}, 349(2):579--603, 1997.

\end{thebibliography}

%%%%%%%%%%%%%%%%%%%%%%%%%%%%%%%%%%%%%%%%%%%%%%%%%%%%%%%%%%%%%%%%%%%%%%

\end{document}